\documentclass[10pt]{amsart}

\usepackage{amsmath}    
\usepackage{graphicx}   
\usepackage{verbatim}   
\usepackage{color}      

\usepackage{amssymb}
\usepackage{amsthm}
\usepackage{mathrsfs}

\usepackage[top=2cm, bottom=2cm, left=2cm, right=2cm]{geometry}

\theoremstyle{theorem}
\newtheorem{Def}{Definition}[section]
\newtheorem{Eg}{Example}
\newtheorem{Prop}[Def]{Proposition}
\newtheorem{Lem}[Def]{Lemma}
\newtheorem{Thm}[Def]{Theorem}

\newtheorem{Ass}[Def]{Assumption}
\theoremstyle{definition}
\newtheorem{Rem}[Def]{Remark}

\newcommand{\e}{\mathbb{E}}
\newcommand{\real}{\mathbb{R}}

\newcommand{\sign}{\text{sign}}
\newcommand{\1}{{\bf 1}}

\newcommand{\F}{\mathcal{F}}
\newcommand{\G}{\mathcal{G}}
\newcommand{\ff}{\mathbb{F}}

\newcommand{\R}{\mathbb{R}}

\allowdisplaybreaks

\usepackage{etoolbox}
\patchcmd{\thebibliography}{\section*}{\section}{}{}

\title{Well-posedness and approximation of some one-dimensional L\'evy-driven non-linear SDEs.}
\date{}

\author{Noufel Frikha}
\address{Noufel Frikha, Universit\'e de Paris, Laboratoire de Probabilit\'es, Statistiques et Mod\'elisation, F-75013 Paris, France}
\email{frikha@math.univ-paris-diderot.fr}

\author{Libo Li}
\address{Libo Li, Department of Mathematics and Statistics, University of New South Wales, Sydney, Australia}
\email{libo.li@unsw.edu.au }

\keywords{McKean-Vlasov SDEs; L\'evy driven SDEs; strong uniqueness; propagation of chaos.\\
 The financial support from the Europlace Institute of Finance is gratefully acknowledged.
}

\begin{document}

\maketitle

\begin{abstract}
In this article, we are interested in the strong well-posedness together with the numerical approximation of some one-dimensional stochastic differential equations with a non-linear drift, in the sense of McKean-Vlasov, driven by a spectrally-positive L\'evy process and a Brownian motion. We provide criteria for the existence of strong solutions under non-Lipschitz conditions of Yamada-Watanabe type without non-degeneracy assumption. The strong convergence rate of the propagation of chaos for the associated particle system and of the corresponding Euler-Maruyama scheme are also investigated. In particular, the strong convergence rate of the Euler-Maruyama scheme exhibits an interplay between the regularity of the coefficients and the order of singularity of the L\'evy measure around zero.  \end{abstract}

\vskip 20pt 
\section{Introduction}\label{introduction}
In this paper, we are interested in the strong well-posedness as well as the numerical approximation of the following one-dimensional non-linear stochastic differential equation (SDE for short) with positive jumps and with dynamics
\begin{align}
X_t & = \xi + \int_0^t b(s, X_s, [X_s]) \, ds + \int_0^t \sigma(s, X_{s}) \, dW_s +  \int_0^t   h(s, X_{s-}) \, dZ_s,  \label{mckean:sde}\\
Z_t & = \int_0^t \int_0^{\infty} z \widetilde{N}(ds, dz) \label{levy:process}
\end{align}

\noindent where $W=(W_t)_{t\geq0}$ is a standard one-dimensional Brownian motion, $\widetilde{N}$ is a compensated Poisson random measure with intensity measure $ds\nu(dz)$ satisfying the condition $\int_0^{\infty} (z\wedge z^2) \nu(dz) < \infty$ and with starting point $\xi$ (with law $\mu$). We here assume that the random variables $\xi$, $W$ and $\widetilde{N}$ are defined on some filtered probability space $(\Omega, \mathcal{F}, (\mathcal{F}_t)_{t\geq 0}, \mathbb{P})$ satisfying the usual conditions and that they are mutually independent. Throughout the article, we will denote by $\mathcal{P}(\R)$ the space of probability measures on $\R$, $\mathcal{P}_q(\R)$ the subset of probability measures in $\mathcal{P}(\R)$ that have finite moment of order $q>0$, equipped with the corresponding Wasserstein distance $W_q$, and by $[\theta]$ the law of a random variable $\theta$.

The SDE \eqref{mckean:sde} generally appears as the mean-field limit of an individual particle evolving within a system of particles, which interact with each other only through the empirical measure of the whole system, when the number of particles goes to infinity. This property is commonly referred in the literature as the propagation of chaos phenomenon. The weak and strong well-posedness in the diffusive setting, i.e. when $h\equiv 0$, have attracted considerable attention of the research community from the last decades to the present day. We  refer the reader to the works of Funaki \cite{Funaki1984}, Oelschlaeger \cite{oelschlager1984}, G\"artner \cite{gartner}, Sznitman \cite{Sznitman}, Jourdain \cite{jourdain:1997} and more recently, Li and Min \cite{Li:min:2}, Chaudru de Raynal \cite{CHAUDRUDERAYNAL}, Mishura and Veretennikov \cite{mishura:veretenikov}, Lacker \cite{la:18}, Chaudru de Raynal and Frikha \cite{chau:frik:18} for a small and incomplete sample. When $Z$ is a square integrable L\'evy process, the strong well-posedness of multi-dimensional non-linear SDEs with time-homogeneous Lipschitz coefficients $\R^d \times \mathcal{P}_2(\R^d) \ni (x, \mu) \mapsto b(x, \mu), \sigma(x, \mu), \, h(x, \mu)$ together with the strong convergence rate of propagation of chaos have been investigated by Jourdain, M\'el\'eard and Woyczynski \cite{jour:mele:woyc:08}. The well-posedness in the weak and strong sense of some multi-dimensional non-linear SDEs driven by non-degenerate symmetric $\alpha$-stable L\'evy processes, $\alpha \in (0,2)$, under some mild H\"older regularity assumptions on the drift and diffusion coefficients with respect to both space and measure variables have been recently established by Frikha, Konakov and Menozzi \cite{Frikha:Konakov:Menozzi}. 

In the absence of the Lipschitz regularity of coefficients and/or non-degeneracy of the underlying noise, it actually turns out to be challenging to establish the well-posedness of such non-linear dynamics and to obtain some quantitative rates of convergence for the propagation of chaos by the related system of particles and for the time discretization by the so-called Euler-Maruyama scheme. Concerning the well-posedness, we again mention the works \cite{CHAUDRUDERAYNAL,chau:frik:18,Frikha:Konakov:Menozzi,jourdain:1997,la:18}. For some quantitative rates of propagation of chaos in the diffusive setting, we mention the recent work of Holding \cite{holding} who studies systems of interacting particles with a constant diffusion coefficient and a drift coefficient given by an H\"older continuous interacting kernel of convolution type are established. Still in the non-degenerate diffusive setting, some new quantitative estimates for propagation of chaos have been recently obtained by Chaudru de Raynal and Frikha \cite{chau:frik:18:b} for McKean-Vlasov SDEs under some mild H\"older regularity assumptions on the drift and diffusion coefficients with respect to both space and measure variables. 

Our main objective here is twofold as it consists specifically in investigating the strong well-posedness and the numerical approximation of the one-dimensional non-linear SDE \eqref{mckean:sde} without any non-degeneracy condition on the underlying noise and only assuming that $\mathbb{R} \times \mathcal{P}_1(\R) \ni(x, \mu)\mapsto b(t, x, \mu)$ is one-sided Lipschitz, uniformly in time, and that both coefficients $\sigma$ and $h$ are H\"older continuous in space, uniformly in time. 

Motivated by the recent developments of continuous-state branching processes, the well-posedness of such dynamics in the linear setting, i.e. when there is no dependence with respect to the measure argument in the drift coefficient, was first addressed by Fu and Li \cite{Fu:Li} and then later extended by Li and Mytnik \cite{Li:Mytnik}. Equations of such type have recently found applications in financial mathematics due to the positivity of the solution and the ability to capture self-exciting effects through the inclusion of the L\'evy component $Z$. 
We refer to the works of Jiao et al. \cite{JMS, JMSS, JMSZ} for the modelling of sovereign interest rate, electricity prices and stochastic volatility. For multi-curve term structure models we refer to Fontana et al. \cite{FGS}. Our results thus allow to consider and to study from both theoretical and numerical point of view some mean-field extension of the models proposed in the aforementioned references. In this direction, it is worth mentioning for instance the recent works of Bo and Capponi \cite{bo:capponi} and of Fouque and Ichiba \cite{fouque:ichiba} for the modelling of systemic risk of a banking system through the lens of interacting particle system driven by independent Brownian motions.   

Concerning the study of the time discretization schemes in the linear framework, we can mention the two recent works of Li and Taguchi \cite{Li:Taguchi:a,Li:Taguchi:b} concerning the strong rate of convergence of the corresponding Euler-Maruyama approximation scheme and of a positivity preserving time discretization scheme. In the spirit of the aforementioned works, we employ the Yamada-Watanabe approximation technique to derive both theoretical and numerical results. Let us however mention that compared to the two recent works \cite{Li:Taguchi:a,Li:Taguchi:b}, we here improve the strong convergence rate of the Euler-Maruyama scheme and we also remove the restrictive boundedness assumption on the diffusion and jump coefficients at the price of some integrability constraints. Let us also mention that such boundedness assumption on the jump coefficient $h$ has appeared persistently in the literature on the numerical approximation of $\alpha$-stable driven SDEs by the Euler-Maruyama scheme, see for example the works of Hashimoto \cite{Ha} and Hashimoto and Tsuchiya \cite{HaTsu}.


The organization of the paper is as follows. The main theorems together with their proofs are presented in Section \ref{main:results}. In Section 3, we mention briefly conditions for which positivity of the solution holds and some potential applications. Section \ref{appendix:section} is an appendix and contains some important but auxiliary results. In particular, a general weak existence result as well as as some moment estimates for some non-linear SDEs of jump-type with continuous coefficients having at most linear growth are presented.  

\section{Assumptions and main results}\label{main:results}
\subsection{Strong solutions to the mean-field SDE}
Let $(\Omega, \F, (\mathcal{F}_t)_{t\geq0}, \mathbb{P})$ be a filtered probability space satisfying the usual conditions. On this probability space, let us consider a standard one-dimensional $(\mathcal{F}_t)_{t\geq0}$-Brownian motion $W$ and a compensated $(\mathcal{F}_t)_{t\geq0}$-Poisson random measure $\widetilde{N}$ with L\'evy measure $\nu$ on $(\R_+, \mathcal{B}(\R_+))$ such that $\int_{\R_+} z\wedge z^2 \, \nu(dz) < \infty$. Suppose that $W$ and $\widetilde{N}$ are independent. Given a real-valued random variable $\xi$ defined on $(\Omega, \F, (\mathcal{F}_t)_{t\geq0}, \mathbb{P})$ and independent of the pair $(W, \widetilde{N})$, we will say that strong existence holds for the non-linear SDE \eqref{mckean:sde} starting from $\xi$ with law $\mu$ at time $0$ if there exists a real-valued and c\`adl\`ag process $X=(X_t)_{t\geq0}$, with marginal law $[X_t]$ at time $t$, adapted to the augmented (and completed) natural filtration $(\mathcal{G}_t)_{t\geq0}$ generated by $W$ and $\widetilde{N}$, that satisfies equation \eqref{mckean:sde} almost surely for all $t\geq0$. We will say that weak existence holds for \eqref{mckean:sde} starting from $\xi$ with law $\mu$ at time $0$ if there exists a filtered probability space $(\Omega, \F, (\mathcal{F}_t)_{t\geq0}, \mathbb{P})$ and a triplet $(X, W, \widetilde{N})$ of $(\mathcal{F}_t)_{t\geq0}$-adapted process such that $X$ is c\`adl\`ag with $[X_0]=\mu$, $W$ is a Brownian motion, $\widetilde{N}$ is a Poisson random measure independent of $W$ with L\'evy  measure $\nu$ and satisfying \eqref{mckean:sde}. We will say that pathwise uniqueness holds for the same equation if for any two solutions $X^1$ and $X^2$ of \eqref{mckean:sde} (possibly defined on two different filtered probability spaces) satisfying $X^1_0=X^2_0= \xi$, $W^1=W^2$ and $\widetilde{N}^1=\widetilde{N}^2$, we have $X^1_t=X^2_t$ almost surely for every $t\geq0$. In particular, note that strong uniqueness implies uniqueness of marginal laws, i.e. $[X^1_t]=[X^2_t]$ for all $t\geq0$, which in turn implies that the non-linear SDE \eqref{mckean:sde} may be regarded as a linear SDE with time-inhomogeneous coefficients.

In the spirit of Li and Mytnik \cite{Li:Mytnik}, we introduce the quantity
\begin{align*}
\alpha_{\nu}:=\inf\{\beta>1: \lim_{x \to 0+} x^{\beta-1} \int_{x}^{\infty}z \, \nu(dz)=0\}
\end{align*}

\noindent which represents the order of singularity of the L\'evy measure at zero. For instance, if $\nu$ is the L\'evy measure of a spectrally positive $\alpha$-stable like process, that is, if $\nu(dz) = \textbf{1}_{(0,\infty)}(z) g(z)/ z^{1+\alpha} \, dz$, with $\alpha \in [1,2]$, $g$ being a non-negative bounded and continuous function on $\R_+$, one has $\alpha_\nu = \alpha$ and the infimum is actually achieved. We recall from Lemma 2.1 of \cite{Li:Mytnik} that $\alpha_\nu \in [1,2]$. Moreover, for any $\alpha>\alpha_{\nu}$, 
\begin{equation}
\lim_{x \to 0+} x^{\alpha-2} \int_{0}^{x} z^2 \nu(dz)=0 \quad \mathrm{and} \quad \lim_{x \to 0+} x^{\alpha-1} \int_{x}^{\infty}z \nu(dz)=0. \label{left:right:tail:asymptotics}
\end{equation}

Note that in our current setting, both limits in \eqref{left:right:tail:asymptotics} are also valid for $\alpha =2$. We now introduce our assumptions on the coefficients $b$, $\sigma$ and $h$ in order to tackle the strong well-posedness of the SDE \eqref{mckean:sde}.\\

\begin{Ass}\label{Ass_1} 
$\,$
\begin{itemize}	
%
	\item[(i)] The coefficients $\R_+ \times \R \times \mathcal{P}_1(\R) \ni (t, x, \mu) \mapsto b(t, x, \mu), \sigma(t, x), \, h(t, x) $ are continuous functions, $\mathcal{P}_1(\R)$ being equipped with the Wasserstein metric $W_1$, and have at most linear growth in $x$ and $\mu$ locally uniformly with respect to $t$, that is, for any $T>0$, there exists some positive constant $C_T$ such that for any $(x, \mu) \in \R \times \mathcal{P}(\R)$
\begin{equation}\label{linear:growth:condition}
 \sup_{0\leq t \leq T} \left\{|b(t, x, \mu)| + |\sigma(t, x)| + |h(t, x)|\right\} \leq C_T (1 + | x |+ W_1(\mu, \delta_0)). 
\end{equation}

	\item[(ii)] There exists some positive constant denoted by $[b]_L$ such that for all $(t, x, y, \mu, \nu) \in  [0,\infty) \times \R^2 \times \mathcal{P}_1(\R)^2$, 
$$
\sign(x-y)(b(t, x, \mu) - b(t, y, \nu)) \leq [b]_L (|x-y| + W_1(\mu, \nu)). 
$$
	
		
	\item[(iii)] For any $t\geq0$, the diffusion coefficient $x\mapsto \sigma(t, x)$ is $\gamma$-H\"older continuous with $\gamma \in [1/2,1]$ and $x\mapsto h(t, x)$ is $\eta$-H\"older continuous with $\eta \in (1-1/{\alpha}_\nu, 1]$, uniformly with respect to $t$. We denote respectively by $[\sigma]_\gamma$ and $[h]_\eta$ the (uniform) H\"older modulus of $\sigma$ and $h$.
	
	\item[(iv)] There exists a positive constant denoted by $[h]_L$ such that for all $t\in [0, \infty)$, for all $x, y \in \R$, 
	$$
	\sign(y-x) (h(t, x) - h(t, y)) \leq [h]_L |x-y|.
	$$
		
\end{itemize}
\end{Ass}

%
%
%
%
\vskip5pt
The assumption (i) is required in order to prove weak existence of solutions to the SDE \eqref{mckean:sde} by employing a compactness argument together with some moment estimates. Though the previous result seems natural, to the best of our knowledge, it is new. The precise statement together with its proof are postponed to the Appendix, see Lemma \ref{weak:solution:and:moment:estimate} in Section \ref{weak:existence:lemma}. The regularity assumptions (ii), (iii) and (iv) allow to prove pathwise uniqueness as well as to study the propagation of chaos for the corresponding system of interacting particles. As already mentioned in the introduction, we will rely on the Yamada-Watanabe approximation technique which is briefly exposed in Section \ref{YW}.  

Before proceeding to the statement of our first main result, let us make some additional comments on Assumption \ref{Ass_1}. Observe that the condition on $b$ holds as soon as $b(t, x, \mu) = b_1(t, x) + b_2(t, x, \mu)$, with $b_1$ non-increasing in space, uniformly in time, and $b_2(t, ., .)$ Lipschitz-continuous on $\R \times \mathcal{P}_1(\R)$, uniformly with respect to time $t$. Our assumption on $h$ holds if $h(t, x) = h_1(t, x) + h_2(t, x)$ with $h_1(t,.)$ Lipschitz continuous (uniformly in time) and $h_2(t, .)$ $\eta$-H\"older continuous (uniformly in time) for some $\eta \in (1-1/\alpha_\nu,1]$ and non-decreasing in space. This last assumption is reminiscent of those introduced in \cite{Li:Mytnik}.

We now present two results concerning weak existence and pathwise uniqueness of solutions for the mean-field SDE \eqref{mckean:sde}.

\begin{Lem}Suppose that assumption \ref{Ass_1} (i) holds and that $\int_{\R} |x|^\beta \mu(dx) + \int_{ z \geq1} z^\beta \nu(dz) < \infty$ for some $\beta>1$. Then, weak existence holds for the SDE \eqref{mckean:sde} starting at time $0$ from the initial point $\xi$ with law $\mu$.
\end{Lem}

\begin{proof}
Weak existence follows from a compactness argument stated in a more general form in Lemma \ref{weak:existence:lemma} of the Appendix in Section \ref{weak:existence:lemma}.
\end{proof}

\begin{Lem}Suppose that assumption \ref{Ass_1} holds. Then, pathwise uniqueness holds for the SDE \eqref{mckean:sde} starting at time $0$ from the initial point $\xi$ with law $\mu \in \mathcal{P}_1(\R)$.
\end{Lem}
\begin{proof} 

Consider two solutions $X$ and $Y$ to \eqref{mckean:sde} with the same input data $(\xi, W, Z)$. We will use the notation $\Delta_t := X_t -Y_t$ and $\gamma_t := \sigma(t, X_t) - \sigma(t, Y_t)$, $\lambda_t := h(t, X_t) - h(t, Y_t)$. Let $\varepsilon \in (0,1)$ and $\delta \in (0,1)$. Using \eqref{phi3} and then applying It\^o's formula, we get
\begin{equation}
\label{first:step:yamada:watanabe:technique:pathwise:uniqueness}
|\Delta_t| \leq \varepsilon + \phi_{\delta, \varepsilon}(\Delta_t) = \varepsilon + {M}_t^{\delta,\varepsilon} +{I}_t^{\delta,\varepsilon} +{J}_t^{\delta,\varepsilon} +{K}_t^{\delta,\varepsilon},
\end{equation}

\noindent where we set
\begin{align*}
	{M}_t^{\delta,\varepsilon}
	:=& \int_{0}^{t} \phi_{\delta,\varepsilon}'(\Delta_{s}) \,  \gamma_s \, dW_{s}+\int_{0}^{t} \int_{0}^{\infty}\left\{\phi_{\delta,\varepsilon}(\Delta_{s-}+\lambda_{s-}z)-\phi_{\delta,\varepsilon}(\Delta_{s-})\right\}
	\widetilde{N}(ds,dz),\\
	{I}_t^{\delta,\varepsilon}
	:=&\int_{0}^{t} \phi_{\delta,\varepsilon}'(\Delta_{s})\{b(s, X_{s}, [X_s])-b(s, Y_s, [Y_s])\}ds,
	\quad\\
	{J}_t^{\delta,\varepsilon}
	:=& \frac{1}{2} \int_{0}^{t} \phi_{\delta,\varepsilon}''(\Delta_{s}) |\gamma_s|^2ds,\\
	{K}_t^{\delta,\varepsilon}
	:=&
	\int_{0}^{t} \int_{0}^{\infty}
	\Big\{
	\phi_{\delta,\varepsilon}(\Delta_{s-} + \lambda_{s-}z)-\phi_{\delta,\varepsilon}(\Delta_{s-}) - \lambda_{s-}z \phi_{\delta,\varepsilon}'(\Delta_{s-})
	\Big\}
	\nu(dz)ds.
\end{align*}

By a standard localization argument, we can define a sequence $(\tau_m)_{m\geq1}$ satisfying $\tau_m \uparrow \infty$ as $m\uparrow \infty$ and such that $(M^{\delta, \varepsilon}_{t\wedge \tau_m})_{t\geq0}$ is an $L^1(\mathbb{P})$-martingale. After taking expectation and performing the above computations, one can finally pass to the limit as $m\uparrow \infty$ using Fatou's lemma and the right-continuity of $(\Delta_t)_{t\geq0}$. Since this procedure is standard, for sake of simplicity, we omit it here and refer the reader to the proof of Theorem 2.2 in \cite{Li:Mytnik} for a similar argument. We now quantify the contribution of ${I}_t^{\delta,\varepsilon}$, ${J}_t^{\delta,\varepsilon}$ and ${K}_t^{\delta,\varepsilon}$.

In order to deal with ${I}_t^{\delta,\varepsilon}$ we make use of Assumption \ref{Ass_1} (iii)
\begin{align*}
\phi_{\delta,\varepsilon}'(\Delta_{s})\{b(s, X_{s}, [X_s])-b(s, Y_s, [Y_s])\} & = |\phi_{\delta,\varepsilon}'(\Delta_{s})| \sign(\Delta_{s}) \{b(s, X_{s}, [X_s])-b(s, Y_s, [Y_s])\} \\
& \leq [b]_L (|\Delta_s| + W_1([X_s], [Y_s])) 
\end{align*}

\noindent so that
$$
{I}_t^{\delta,\varepsilon} \leq [b]_L \int_0^t (|\Delta_s| + W_1([X_s], [Y_s])) \, ds.
$$

From Assumption \ref{Ass_1} (iii) and \eqref{phi4}, we directly get
$$
{J}_t^{\delta,\varepsilon} \leq  [\sigma]^2_\gamma\int_{0}^{t} \frac{|\Delta_s|^{2\gamma}}{|\Delta_s| \log(\delta)} \textbf{1}_{[\varepsilon/\delta, \varepsilon]}(|\Delta_s|) ds \leq  [\sigma]^2_\gamma \frac{\varepsilon^{2\gamma-1}}{\log(\delta)} t.
$$

In order to deal with ${K}_t^{\delta,\varepsilon}$, we write it as the sum of ${K}_t^{1, \delta,\varepsilon}$ and ${K}_t^{2, \delta,\varepsilon}$ with
\begin{align*}
{K}_t^{1, \delta,\varepsilon} :=& \int_{0}^{t} \int_{0}^{\infty} \Big\{ \phi_{\delta,\varepsilon}(\Delta_{s-} + \lambda_{s-}z)-\phi_{\delta,\varepsilon}(\Delta_{s-}) - \lambda_{s-}z \phi_{\delta,\varepsilon}'(\Delta_{s-}) \Big\} \, \textbf{1}_{\left\{ \Delta_{s-} \lambda_{s-} >0 \right\}}\nu(dz)ds, \\
{K}_t^{2, \delta,\varepsilon} :=& \int_{0}^{t} \int_{0}^{\infty} \Big\{ \phi_{\delta,\varepsilon}(\Delta_{s-} + \lambda_{s-}z)-\phi_{\delta,\varepsilon}(\Delta_{s-}) - \lambda_{s-}z \phi_{\delta,\varepsilon}'(\Delta_{s-}) \Big\} \textbf{1}_{\left\{ \Delta_{s-} \lambda_{s-} <0 \right\}} \nu(dz)ds.
\end{align*}
To deal with ${K}_t^{1, \delta,\varepsilon}$, we apply Lemma \ref{key_lem_0} with $y = \Delta_{s-}$, $x=\lambda_{s-}$ and $u>0$ to be chosen later. From Assumption \ref{Ass_1} (iii), we obtain
\begin{align}
\int_{0}^{\infty} &  \Big\{ \phi_{\delta,\varepsilon}(\Delta_{s-} + \lambda_{s-}z)-\phi_{\delta,\varepsilon}(\Delta_{s-}) - \lambda_{s-}z \phi_{\delta,\varepsilon}'(\Delta_{s-}) \Big\} \, \textbf{1}_{\left\{ \Delta_{s-} \lambda_{s-} >0 \right\}}\nu(dz) \nonumber \\
& \leq 2   \textbf{1}_{\left\{ \Delta_{s-} \lambda_{s-} >0 \right\}}  \1_{(0,\varepsilon]}(|\Delta_{s-}|)
	\left\{
	\frac{|\lambda_{s-}|^2}{|\Delta_{s-}|\log \delta}  \int_{0}^{u} z^2 \nu(dz)
	+ |\lambda_{s-}| \int_{u}^{\infty} z \nu(dz)
	\right\} \nonumber \\
	& \leq \frac{2 [h]^2_\eta}{\log(\delta)}  \varepsilon^{2 \eta-1}   \int_{0}^{u} z^2 \nu(dz) + [h]_\eta \varepsilon^{\eta} \int_{u}^{\infty} z \nu(dz). \label{intermediate:bound}
 \end{align}


To proceed, we aim at selecting a sufficiently small $u$, which equilibrates the two terms that appear in the above upper-estimate. Having in mind \eqref{left:right:tail:asymptotics}, we observe that the upper-bound \eqref{intermediate:bound} can be further bounded as follows
	\begin{align*}
	& \frac{2 [h]^2_\eta}{\log(\delta)}  \varepsilon^{2 \eta-1}   \int_{0}^{u} z^2 \nu(dz) + [h]_\eta \varepsilon^{\eta} \int_{u}^{\infty} z \nu(dz) \\
	& \leq \frac{2 [h]^2_\eta}{\log(\delta)}  \frac{\varepsilon^{2 \eta-1}}{u^{\alpha-2}} u^{\alpha-2}  \int_{0}^{u} z^2 \nu(dz) + [h]_\eta \frac{\varepsilon^{\eta}}{u^{\alpha-1}} u^{\alpha-1} \int_{u}^{\infty} z \nu(dz) \\
	&\leq \bigg[\frac{2 [h]^2_\eta}{\log(\delta)}  \frac{\varepsilon^{2 \eta-1}}{u^{\alpha-2}}  + [h]_\eta \frac{\varepsilon^{\eta}}{u^{\alpha-1}}\bigg] (I^{\alpha}_{\varepsilon}+J^{\alpha}_\varepsilon)
	= (2[h]^2_\eta+[h]_\eta)\frac{\varepsilon^{1-\alpha(1-\eta)}}{\log(\delta)^{\alpha-1}} (I^{\alpha}_{\varepsilon}+J^{\alpha}_\varepsilon)\notag
	\end{align*}

\noindent where in the last equality we have chosen $u$ so that $\frac{\varepsilon^{2\eta-1}}{\log(\delta)u^{\alpha-2}} = \frac{\varepsilon^\eta}{u^{\alpha-1}}$, that is, $u=\log(\delta) \varepsilon^{1-\eta}$ and introduced the two quantities 
\begin{align*}
I^{\alpha}_\varepsilon := (\log(\delta) \varepsilon^{1-\eta})^{\alpha-2}  \int_{0}^{\log(\delta) \varepsilon^{1-\eta}} z^2 \nu(dz) \quad \mathrm{and}\quad J^{\alpha}_\varepsilon := (\log(\delta) \varepsilon^{1-\eta})^{\alpha-1} \int_{\log(\delta)\varepsilon^{1-\eta}}^{\infty} z \nu(dz)
\end{align*}
for a fixed $\alpha \in (\alpha_\nu, 1/(1-\eta))$. Note that this choice of $\alpha$ is admissible since $\eta \in (1-1/\alpha_\nu,1]$. If $\eta < 1$, from \eqref{left:right:tail:asymptotics}, there exists $\varepsilon_0:=\varepsilon_0(\alpha, \delta)>0$ such that $I^{\alpha}_{\varepsilon}+J^{\alpha}_\varepsilon \leq1$ for any $\varepsilon \in (0,\varepsilon_0)$. If $\eta=1$, we select $\delta=2$ and from Assumption \ref{Ass_1}, one may bound the quantity $I^{\alpha}_{\varepsilon}+J^{\alpha}_\varepsilon$ by $2 \int_0^\infty z \wedge z^2 \nu(dz)$.  From the above computation, for any $\alpha \in (\alpha_\nu, 1/(1-\eta))$, there exists a positive $\varepsilon_0$ such that for any $\varepsilon \in (0, \varepsilon_0)$
$$
{K}_t^{1, \delta,\varepsilon}  \leq C\frac{\varepsilon^{1-\alpha(1-\eta)}}{\log(\delta)^{\alpha-1}} .
$$

In order to deal with $K^{2, \delta, \varepsilon}_t$, we remark that on the set $\left\{ \Delta_{s-} \lambda_{s-} < 0\right\}$, from assumption \ref{Ass_1} (iv), one has $0 < -\sign(\Delta_{s-}) \lambda_{s-} = |\lambda_{s-}| \leq [h]_L |\Delta_{s-}|$ so that separating the $\nu(dz)$-integral into the two domains $  z \in (0, \frac{|\Delta_{s-}|}{2|\lambda_{s-}|}\wedge 1]$ and $z \in (\frac{|\Delta_{s-}|}{2|\lambda_{s-}|}\wedge 1, \infty)$, we get
$$
\forall z \in \bigg(0, \frac{|\Delta_{s-}|}{2|\lambda_{s-}|}\wedge 1\bigg], \quad \phi_{\delta,\varepsilon}(\Delta_{s-} + \lambda_{s-}z)-\phi_{\delta,\varepsilon}(\Delta_{s-}) - \lambda_{s-}z \phi_{\delta,\varepsilon}'(\Delta_{s-})  \leq C  \1_{(0,\varepsilon]}(|\Delta_{s-}|)	\frac{|\lambda_{s-}|^2}{|\Delta_{s-}|\log \delta}z^2
$$

\noindent and
$$
\forall z \in \bigg(\frac{|\Delta_{s-}|}{2|\lambda_{s-}|}\wedge 1, \infty\bigg), \quad |\phi_{\delta,\varepsilon}(\Delta_{s-} + \lambda_{s-}z)-\phi_{\delta,\varepsilon}(\Delta_{s-})| + |\lambda_{s-}z \phi_{\delta,\varepsilon}'(\Delta_{s-})|  \leq C |\lambda_{s-}| z\leq C |\Delta_{s-}| z.
$$

Combining the two previous bounds, we obtain
\begin{align}
\int_{0}^{\infty} &  \Big\{ \phi_{\delta,\varepsilon}(\Delta_{s-} + \lambda_{s-}z)-\phi_{\delta,\varepsilon}(\Delta_{s-}) - \lambda_{s-}z \phi_{\delta,\varepsilon}'(\Delta_{s-}) \Big\} \, \textbf{1}_{\left\{ \Delta_{s-} \lambda_{s-} <0 \right\}}\nu(dz) \nonumber \\
& \leq C   \1_{(0,\varepsilon]}(|\Delta_{s-}|)	\frac{|\lambda_{s-}|^2}{|\Delta_{s-}|\log \delta} \int_0^{ \frac{|\Delta_{s-}|}{2|\lambda_{s-}|}\wedge 1} z^2 \nu(dz) +C |\Delta_{s-}|  \int_{ \frac{|\Delta_{s-}|}{2|\lambda_{s-}|}\wedge 1}^{\infty} z \nu(dz)  \nonumber\\
& \leq C  \left\{ \frac{\varepsilon}{\log(\delta)}  +  |\Delta_{s-}| \right\}. \label{second:part:Kt:1}
 \end{align}

Hence, 
$$
K^{2, \delta, \varepsilon}_t \leq C  \left\{ \frac{\varepsilon}{\log(\delta)} t  +  \int_0^t |\Delta_{s}| \,ds \right\}.
$$

Taking expectation in \eqref{first:step:yamada:watanabe:technique:pathwise:uniqueness} and gathering the above estimates we obtain
\begin{equation}
\label{last:step:before:gronwall:pathwise:uniqueness:mckeanvlasov}
\mathbb{E}[|\Delta_t|] \leq C \bigg[\varepsilon \bigg(1+\frac{1}{\log(\delta)}\bigg)+   \int_0^t (\mathbb{E}[|\Delta_s|] + W_1([X_s], [Y_s])) \, ds +  \frac{\varepsilon^{2\gamma-1}}{\log(\delta)}  + \frac{\varepsilon^{1-\alpha(1-\eta)}}{\log(\delta)^{\alpha-1}} \bigg]
\end{equation}

\noindent for any $\alpha \in (\alpha_\nu, 1/(1-\eta))$ and for any $\varepsilon \in (0, \varepsilon_0)$. Hence, passing to the limit as $\varepsilon \downarrow 0$ and then using Gronwall's lemma yield
$$
\mathbb{E}[|\Delta_t|] \leq C \int_0^t W_1([X_s], [Y_s])) \, ds
$$

By the very definition of the Wasserstein metric of order one, one has $W_1([X_t], [Y_t]) \leq \mathbb{E}[|\Delta_t|] $ which combined with the previous inequality allows us to conclude the proof.
\end{proof}

Combining the two previous lemmas with Yamada-Watanabe theorem, we thus obtain our first main result.

\begin{Thm}\label{strong:existence:uniqueness} Suppose that assumption \ref{Ass_1} holds and that $\int_{\R} |x|^\beta \mu(dx) + \int_{ z \geq1} z^\beta \nu(dz) < \infty$ for some $\beta>1$. Then, there exists a unique strong solution to the SDE \eqref{mckean:sde} starting at time $0$ from the initial point $\xi$ with law $\mu$.

\end{Thm}

\begin{Rem} We recall that in the case of $\alpha$-stable like L\'evy measure, that is, $\nu(dz) = \textbf{1}_{(0,\infty)}(z)g(z)/ z^{1+\alpha} \, dz$, with $\alpha \in (1,2]$, $g$ being a non-negative bounded and continuous function on $\R_+$, one has $\alpha_\nu = \alpha$. In this case, one may slightly weaken assumption \ref{Ass_1} (iii) by letting $\eta \in [1-1/\alpha_\nu,1]$, if $\alpha \in (1,2]$. Indeed, if $\eta = 1-1/\alpha_\nu$ with $\alpha_\nu\in (1,2]$, then the last term appearing in the right-hand side of \eqref{last:step:before:gronwall:pathwise:uniqueness:mckeanvlasov} becomes $\log(\delta)^{1-\alpha_\nu}$ so that, after passing to the limit as $\varepsilon\downarrow 0$, one may conclude by letting $\delta \uparrow \infty$.
\end{Rem}

\subsection{Approximation of the mean-field limit dynamics by a system of interacting particles}

In this section we consider the approximation of the dynamics \eqref{mckean:sde} by the corresponding system of particles. For a positive integer $N$, let us introduce a sequence $(\xi^{i}, W^{i}, Z^{i})_{1\leq i \leq N}$ of i.i.d. copies of $(\xi, W, Z)$ which is assumed to be defined on the same probability space $(\Omega, \mathcal{F}, \mathbb{P})$ for sake of simplicity. The system of interacting particles $\left\{X^{i, N}_t, 1\leq i\leq N, t\geq0 \right\}$ is defined by the following $N$-dimensional SDE with dynamics
\begin{equation}
\label{particle:system}
X^{i, N}_t = \xi^{i} + \int_0^t b(s, X^{i, N}_{s}, \mu^{N}_s) \, ds + \int_0^t \sigma(s, X^{i,N}_{s}) \, dW^{i}_s +  \int_0^t   h(s, X^{i,N}_{s-}) \, dZ^{i}_s, \quad 1\leq i \leq N
\end{equation}

\noindent where $\mu^{N}_t := N^{-1}\sum_{i=1}^{N} \delta_{X^{i, N}_t}$, $t\geq0$, is the empirical measure associated to \eqref{particle:system} taken at time $t$.

Let us point out that even if the above system of particles appears as the natural candidate for the approximation of the mean-field SDE \eqref{mckean:sde}, at the moment it is not clear if it is well-defined. Our first aim here is to investigate the well-posedness in the strong sense of \eqref{particle:system}. We then provide an error bound for the $L^{1}$-distance between $X^{i, N}$ and the dynamics $\bar{X}^{i, N}$ constructed as i.i.d. copies of the limit equation \eqref{mckean:sde} with the same input $(\xi^{i}, W^{i}, Z^{i})_{1\leq i \leq N}$ as the system \eqref{particle:system}, namely
\begin{equation}
\label{coupling:particle:system}
\bar{X}^{i, N}_t = \xi^{i} + \int_0^t b(s, \bar{X}^{i, N}_{s}, \mu^{i}_s) \, ds + \int_0^t \sigma(s, \bar{X}^{i,N}_{s}) \, dW^{i}_s + \int_0^t h(s, \bar{X}^{i, N}_{s-}) \, dZ^{i}_s, \quad 1\leq i \leq N
\end{equation}

\noindent where $\mu^{i}_t = [\bar{X}^{i,N}_t]$, $t\geq0$. Note that by weak uniqueness of solutions to the SDE \eqref{coupling:particle:system} and interchangeability in law of the $\xi^{i}$ satisfying $[\xi^{i}]=\mu$ one has $\mu^{i}_t = \mu_t$, for any $t\geq0$ and for any $i \in \left\{ 1, \cdots, N\right\}$, so that the system of particles $(\bar{X}^{i, N})_{1\leq i \leq N}$ indeed corresponds to i.i.d. copies of the SDE \eqref{mckean:sde}.  

\begin{Thm}\label{prop:wellposedness:particle:system}
Under assumption \ref{Ass_1}, the SDE \eqref{particle:system} admits a unique strong solution for any initial distribution $\mu \in \mathcal{P}_1(\mathbb{R})$. 

Assume additionally that $\int_{z\geq1} z^{\beta} \, \nu(dz) < \infty$ for some $\beta >1$, $\beta\neq 2$, and that the initial distribution $\mu \in \mathcal{P}_1(\R)$ of the SDE \eqref{mckean:sde} has a finite $\beta$-moment, that is, $M_\beta(\mu)=\int_{\R}|x|^\beta \mu(dx) < \infty$. Then, for any positive integer $N$, for any $T>0$, one has
\begin{equation}
\max_{1\leq i\leq N} \sup_{0\leq t \leq T} \mathbb{E}[|X^{i, N}_t- \bar{X}^{i,N}_t|] + \sup_{0\leq t \leq T}\mathbb{E}[W_1(\mu^{N}_t , \mu_t)] \leq C(N^{-1/2} + N^{-(\beta-1)/\beta}) \label{convergence:rate:particle:system}
\end{equation}

\noindent for some positive constant $C$ depending only on $T$, $\beta$ and $M_\beta(\mu)$. 
\end{Thm}

\begin{proof}
\noindent\emph{Step 1: Weak existence.} 
First let us note that under assumption \ref{Ass_1} the maps $\mathbb{R}_+\times \mathbb{R}^N \ni (t, x) \mapsto b(t, x_i, \mu^{N}_{x}), \, \sigma(t, x_i), \, h(t, x_i)$, with $\mu^{N}_x := N^{-1} \sum_{i=1}^{N} \delta_{x_i}$, are continuous with at most linear growth for any $1\leq i\leq N$. Indeed, for any fixed $N\geq1$, if $(t^n, x^{n})_{n\geq1}$ converges to $(t, x) \in \R_+\times \R^N$, then $\lim_{n} W_1(\mu^{N}_{x^n}, \mu^N_x) \leq \lim_{n} N^{-1} \sum_{i=1}^{N} |x^{n}_i - x_i| = 0$. Therefore, the triple $(b(t_n, x^{n}_i, \mu^{N}_{x^n}),  \sigma(t_n, x^{n}_i), h(t_n, x^{n}_i))_{n\geq1}$ converges to the triple $(b(t, x_i, \mu^{N}_x) ,\sigma(t, x_i), h(t, x_i))$. In the spirit of step 2 in the proof of Lemma \ref{weak:solution:and:moment:estimate}, we then introduce the sequence $(X^{(m)}= (X^{i, (m)})_{1\leq i \leq N})_{m\geq1}$ of SDEs with dynamics

 \begin{equation}
 X^{i, (m+1)}_t  = \xi + \int_0^t b(s, X^{i, (m)}_{s}, \mu^{N}_{X^{(m)}_s}) \, ds + \int_0^t \sigma(s, X^{i, (m)}_{s}) \, dW_s + \int_0^t  h(s, X^{i, (m)}_{s-}) \, dZ^{i, m}_{s}, \quad 1\leq i\leq N \label{approximation:particle:system:sde}
 \end{equation} 
\noindent where $Z^{i, m}_t = \int_0^t \int_{\R\backslash{\left\{0\right\}}} z \widetilde{N}^{i}_{m}(ds,dz)$, $\widetilde{N}^{i}_m$, $1\leq i \leq N$, being $N$-independent compensated Poisson random measures on $[0,\infty)\times \R\backslash{\left\{0\right\}}$ with intensity measure $dt \textbf{1}_{|z|\leq m} \nu(dz)$. Following similar lines of reasonings as those employed in the first step of Lemma \ref{weak:solution:and:moment:estimate}, one may prove that for any weak solution to \eqref{approximation:particle:system:sde} with starting distribution $\mu \in \mathcal{P}_1(\R)$, for any $T>0$, one has
\begin{equation*}
\sup_{m\geq1} \max_{1\leq i  \leq N}\mathbb{E}[\sup_{0\leq t \leq T}|X^{i, (m)}_t|] < \infty.
\end{equation*}

Similarly, relabelling the indices if necessary, one may assume that the sequence $(X^{i, (m)}, Z^{i, m}, 1\leq i \leq N)_{m\geq1}$ converges in law to $(X^{i}, Z^{i}, 1\leq i \leq N)$ in $\mathcal{D}([0,\infty), \R^{N}\times \R^{N})$. The sequence $(Z^{i, (m)}, 1\leq i \leq N)_{m\geq1}$ also satisfies the P-UT property since
$$
\max_{1\leq i \leq N} \mathbb{E}[\sup_{s \in [0,t]}|\Delta Z^{i, m}_s|] \leq 1 + C t
$$
\noindent with $C:= \int_{z \geq1} z \nu(dz) < \infty$. Finally, in a completely analogous manner as in step 2 of the proof of Lemma \ref{weak:solution:and:moment:estimate}, since the maps $[0,\infty) \times \R^N \ni(t, x) \mapsto b(t, x_i, \mu^N_x), \, \sigma(t, x_i),\, h(t, x_i)$ are continuous, from the continuous mapping theorem, the family 
$$
(X^{i, (m+1)}_t, b(t, X^{i, (m+1)}_t, \mu^{N}_{X^{(m)}_t}), \sigma(t, X^{i, (m+1)}_t), h(t, X^{i, (m+1)}_t), W^i_t, Z^{i, m}_t, 1\leq i \leq N)_{t\geq0}, \, m\geq0
$$ 
\noindent converges in law to $(X^{i}_t, b(t, X^{i}_t, \mu^{N}_{X_t}), \sigma(t, X^{i}_t), h(t, X^{i}_t), W^{i}_t, Z^{i}_t, 1\leq i \leq N)_{t\geq0}$ in $\mathcal{D}([0,\infty), (\R^{6})^N)$. Thus, passing to the limit in the dynamics \eqref{approximation:particle:system:sde}, we deduce that there exists a weak solution to the SDE \eqref{particle:system}. It thus suffices to prove pathwise uniqueness.\\

\noindent\emph{Step 2: Pathwise uniqueness.} 
Let us consider two weak solutions $X^{N}:=(X^{i, N}, 1\leq i \leq N)$ and $(Y^{i, N}, 1\leq i \leq N)$ of \eqref{particle:system} with the same input $(\xi^{i}, W^{i}, Z^{i})_{1\leq i\leq N}$. Following exactly the same lines of reasonings as those employed in the proof of Theorem \ref{strong:existence:uniqueness}, introducing similarly the quantities $\Delta^{i}_t := X^{i, N}_t - Y^{i, N}_t$ for $i=1, \cdots, N$ and $\nu^{N}_t = N^{-1} \sum_{i=1}^{N} \delta_{Y^{i, N}_t}$, instead of \eqref{last:step:before:gronwall:pathwise:uniqueness:mckeanvlasov} we get
\begin{align*}
\mathbb{E}[|\Delta^{i}_t|] \leq C\bigg[\varepsilon +   \int_0^t \bigg(\mathbb{E}[|\Delta^{i}_s|] + \mathbb{E}\left[W_1(\mu^{N}_s, \nu^N_s)\right]\bigg) \, ds +  \varepsilon^{2\gamma-1} + \varepsilon^{1-\alpha(1-\eta)} \bigg]
\end{align*}

\noindent for any $\alpha \in (1-1/\alpha_\nu,1/(1-\eta))$.
Passing to the limit as $\varepsilon \downarrow 0$, then summing over $i$ and finally using the standard inequality $W_1(\mu^{N}_t, \nu^N_t) \leq N^{-1}\sum_{i=1}^{N} |\Delta^{i}_t|$, we get
$$
\frac1N \sum_{i=1}^{N} \mathbb{E}[|\Delta^{i}_t|] \leq C\int_0^t \frac1N\sum_{i=1}^{N} \mathbb{E}[|\Delta^{i}_s|] \, ds
$$
 \noindent so that, by Gronwall's lemma, we deduce that $\mathbb{E}[W_1(\mu^{N}_t, \nu^N_t)]=0$ for all $t\geq 0$. Hence, pathwise uniqueness holds for the SDE \eqref{particle:system} so that, by the Yamada-Watanabe theorem, it has a unique strong solution.

\bigskip

\noindent\emph{Step 3: Propagation of chaos.} 
In order to prove \eqref{convergence:rate:particle:system}, we again follow the lines of reasoning of the proof of Theorem \ref{strong:existence:uniqueness}. Namely, introducing the quantity $\bar{\Delta}^{i}_t := X^{i, N}_t - \bar{X}^{i, N}_t$ for $i=1, \cdots, N$ and $\bar{\mu}^{N}_t = N^{-1} \sum_{i=1}^{N} \delta_{\bar{X}^{i, N}_t}$, instead of \eqref{last:step:before:gronwall:pathwise:uniqueness:mckeanvlasov} we get
\begin{equation*}
\mathbb{E}[|\bar{\Delta}^{i}_t|] \leq C \bigg[\varepsilon +   \int_0^t \bigg(\mathbb{E}[|\bar{\Delta}^{i}_s|] + \mathbb{E}\left[W_1(\mu^{N}_s, \mu_s)\right]\bigg) \, ds +  \varepsilon^{2\gamma-1} + \varepsilon^{1-\alpha(1-\eta)}\bigg].
\end{equation*}

By the triangle inequality 
\begin{equation}
W_1(\mu^{N}_s, \mu_s) \leq W_1(\mu^{N}_s, \bar{\mu}^{N}_s) + W_1(\bar{\mu}^{N}_s, \mu_s) \leq \frac{1}{N}\sum_{i=1}^{N} | \bar{\Delta}^{i}_s| + W_1(\bar{\mu}^{N}_s, \mu_s) \label{estimate:triangular:inequality:particle:system}
\end{equation}

\noindent and noticing that the processes $((X^{i,N}, \bar{X}^{i,N}))_{1\leq i \leq N}$ are identically distributed yield 
 \begin{equation*}
\mathbb{E}[|\bar{\Delta}^{1}_t|] \leq C \bigg[\varepsilon +   \int_0^t \bigg(\mathbb{E}[ |\bar{\Delta}^{1}_s|] + \mathbb{E}\left[W_1(\bar{\mu}^{N}_s, \mu_s)\right]\bigg) \, ds + \varepsilon^{2\gamma-1}  + \varepsilon^{1-\alpha(1-\eta)}\bigg].
\end{equation*}

Applying Gronwall's lemma and then letting $\varepsilon \downarrow 0$ we finally get
\begin{equation}
\mathbb{E}[|\bar{\Delta}^{1}_t|] \leq C \int_0^t \mathbb{E}\left[W_1(\bar{\mu}^{N}_s, \mu_s)\right] \, ds. \label{estimate:gronwall:lemma:particle:system}
\end{equation}

We now discuss the rate of convergence stated in \eqref{convergence:rate:particle:system} under the additional assumption that the L\'evy measure satisfies $\int_{z\geq1}z^\beta \nu(dz) < \infty$ and that the initial distribution $\mu$ has a finite moment of order $\beta$, for some $\beta>1$. Now, from Lemma \ref{weak:solution:and:moment:estimate} (i), it holds
$$
\max_{1\leq i \leq N}\mathbb{E}[\sup_{0\leq t\leq T} |\bar{X}^{i,N}_t|^\beta] < \infty.
$$

It then follows from Theorem 4 in Fournier and Guillin \cite{Fournier:Guillin} that there exists some positive constant $C$ only depending on $\beta$ such that
$$
\mathbb{E}[W_1(\bar{\mu}^N_t, \mu_t) ] \leq C \mathbb{E}[|X_t|^q]^{1/q} (N^{-1/2} + N^{-(\beta-1)/\beta}), \quad \beta\neq 2.
$$ 

Taking the supremum over $t\in[0,T]$ and plugging the above bound into \eqref{estimate:gronwall:lemma:particle:system} we firstly get the desired upper-bound for the quantity $\max_{1\leq i\leq N} \sup_{0\leq t\leq T} \mathbb{E}[|\bar{\Delta}^{i}_t|]$. Then, from \eqref{estimate:triangular:inequality:particle:system}, we derive the similar estimate for the quantity $ \sup_{0\leq t \leq T}\mathbb{E}[W_1(\mu^{N}_t , \mu_t)]$. The proof is now complete. 
\end{proof}

\subsection{Euler-Maruyama time-dscretization scheme for the system of particles}
In the previous section, we established a strong rate of convergence of propagation of chaos for the system of particles associated to the McKean-Vlasov SDE \eqref{mckean:sde}. From a numerical perspective, the system of particles is not tractable and one usually has to approximate the dynamics of the system \eqref{particle:system} by considering the so-called Euler-Maruayama time discretization scheme that we now introduce and analyze. To facilitate our computations, we here only consider the time-homogeneous setting and we claim that, given appropriate regularity assumptions on the maps $t \mapsto b(t, x, \mu)$, $t \mapsto \sigma(t, x)$ and $t \mapsto h(t, x)$, the time non-homogeneous case could be dealt with from similar lines of reasonnings. 

For a given finite time horizon $T>0$ and a positive integer $n$, let us introduce the equally spaced time grid on the interval $[0,T]$, given by $0= t_0 < t_1 < t_2< \dots< t_n = T$, where $t_k = k \delta$ and $\delta = T/n$. We define $\eta(t) = t_k$ for $t \in (t_k, t_{k+1}]$, for $k=0, \cdots, n-1$, $\eta(0) = -\infty$ and $\eta(t) = T$ for $t> T$. The Euler-Maruyama approximation scheme $(X^{n, i, N})_{t\in [0,T]}$, $i=1, \dots, N$, of the system of particles \eqref{particle:system} is given by the following dynamics
\begin{equation}
X^{n,i,N}_t  = \xi^{i} + \int_0^t b(X^{n, i, N}_{\eta(s)}, \mu^{n,N}_{\eta(s)}) ds + \int_0^t \sigma(X_{\eta(s)}^{n, i, N}) \, dW^{i}_s + \int_0^t h(X_{\eta(s)}^{n,i,N}) \, dZ^{i}_s \label{euler:approximation:scheme}
\end{equation}

\noindent where we set $\mu^{n, N}_s := N^{-1}\sum_{i=1}^{N} \delta_{X^{n, i, N}_s}$. We also introduce the integrability index of the tail of the L\'evy measure
\begin{align*}
\beta_\nu & = \sup\{\beta\leq 2: \int^\infty_1 z^\beta \nu(dz) <\infty\}. 
\end{align*}
Note that the above supremum does exist and satisfies $\beta_\nu \geq 1$ since $\int_{1}^{\infty} z\,\nu(dz) < \infty$. 
In the case of $\alpha$-stable like L\'evy measure, with index $\alpha \in [1,2]$, we have $\beta_\nu  = \alpha_\nu = \alpha$ and in the case of tempered $\alpha$-stable L\'evy measure, we have $\beta_\nu  = 2$ and $\alpha_\nu = \alpha$. In order to derive the strong $L^{1}(\mathbb{P})$ convergence rate of the Euler-Maruyama scheme, we introduce the following additional assumptions on the coefficients and the L\'evy measure: 
\begin{Ass} \label{Ass_2}$\,$
\begin{itemize}
\item[(i)] For any $\mu\in \mathcal{P}_1(\R)$, the map $x\mapsto b(x, \mu)$ is $\rho$-H\"older continuous, uniformly in $\mu$, for some $\rho \in (0,1]$. Namely, there exists some positive constant $\kappa$ such that for any $\mu \in \mathcal{P}_1(\R)$:
$$
|b(x, \mu) - b(y, \mu)| \leq \kappa |x-y|^\rho. 
$$

\item [(ii)] $\gamma \in [1/2,\beta_\nu /2)$.
\item [(iii)] $\beta_\nu  > \eta \alpha_\nu$.
\end{itemize}
\end{Ass}
The assumptions (ii) and (iii) stem from some integrability constraints when one investigates the convergence rate of the Euler-Maruyama approximation schemes. Let us note that for $\alpha$-stable like L\'evy measure with index $\alpha \in (1,2]$, the assumption (ii) imposes $\gamma \in [1/2, \alpha/2)$ and (iii) imposes $\eta < 1$ while for tempered $\alpha$-stable like L\'evy measure with index $\alpha \in [1,2]$, the assumption (ii) imposes $\gamma \in [1/2, 1)$ and (iii) is always satisfied if $\alpha \neq 2$ and imposes $\eta<1$ for $\alpha=2$. 

Before stating the main result of this section, we start with the following preparatory lemmas. 

\begin{Lem}\label{euler:moment} Let $S$ be a $\ff$-stopping time taking values in the interval $[0,T]$ then
\begin{itemize}
\item[(i)] the random time $\tau(S) = \inf\{ s\geq 0 : \eta(s) \geq S\}$ is a $\ff$-stopping time and $\{\tau(S) \geq t\} = \{\eta(t)< S\}$,
\item[(ii)] for any $\ff$-adapted process $X$, we have $X_{\eta(S)}$ and $\eta(S)$  are both $\F_{S-}$ measurable.
\end{itemize}
\begin{proof} (i) Using the fact that $\eta$ is left-continuous, for any it holds
\begin{align*}
\{\tau(S) \geq t\} = \{\eta(t)< S\} =  \left(\bigcup_{i=0}^{n-1}\{t_i < t \leq t_{i+1} \} \cap \{t_i < S \}\right) \cup \Big(\{T< t\} \cap \{T < S\}\Big) \in \F_{t-}
\end{align*}
(ii) For any $c\in \mathbb{R}$, using the fact that $S$ is a $\ff$-stopping time taking value in $[0,T]$ we have
\begin{align*}
\{X_{\eta(S)} < c \} =  \bigcup_{i=0}^{n-1}\{t_i < S \leq t_{i+1} \} \cap \{X_{t_i} < c \} \in \F_{S-}.
\end{align*}
Similarly, we see that $\eta(S)$ is $\F_{S-}$ measurable.
\end{proof}
\end{Lem}

\begin{Lem} \label{EM_esti_1} Under the same assumptions as in the statement of Lemma \ref{weak:solution:and:moment:estimate} (i), for any $T>0$, there exists a positive constant $C_T$ such that for all positive integer $n$
\begin{align*}
\max_{1\leq i \leq N} \e\Big[\sup_{t\in [0, T]} |X^{n, i, N}_{t}|^\beta\Big] \leq C_T \qquad \mathrm{and} \qquad 
\max_{1\leq i \leq N}\sup_{0\leq t \leq T}\mathbb{E}\Big[|{X}_{t}^{n, i, N}-{X}_{\eta(t)}^{n, i, N}|^\beta\Big] \leq C_Tn^{-\frac{\beta}{2}}
\end{align*}
where we recall that $\eta(t) = t_k$ for any $t \in (t_k, t_{k+1}]$ and for any $k=0, \cdots, n-1$.
\end{Lem}
\begin{proof}
\noindent The proof of the first moment estimate follows from similar lines of reasonings as those employed to prove \eqref{moment:estimate:jump:type:mckean:sde} of Lemma \ref{weak:solution:and:moment:estimate}. We focus below on the particularities which are induced by having piecewise constant coefficients. 
\vskip3pt

\noindent \emph{Step 1:} For notational convenience, in the proof of the first moment estimate, we write $X^{i}:= (X^{n, i, N}_t)_{0\leq t \leq T}$, $\mu_t = \mu^{n, N}_t$, $t \in [0,T]$, $W^{i}= (W^{i}_t)_{0\leq t \leq T}$, $\widetilde{Z}^{i} = (\widetilde{Z}^{i}_t)_{0\leq t\leq T}$ and $\widehat{Z}^{i} = (\widehat{Z}^i)_{0\leq t \leq T}$ where we introduced the notations $\widetilde{Z}^{i}_t = \int_0^t \int_0^1 z d\widetilde{N}^{i}(ds,dz)$ and $\widehat{Z}^{i}_t = \int_0^t \int_1^{\infty} z d\widetilde{N}^{i}(ds, dz)$. From the Yamada-Watanabe theorem, there exists a measurable map 
$$
\Phi_N: (\R^d)^{N} \times \bigg(\mathcal{C}([0,T]; \R^d)\bigg)^{N} \times \bigg(\mathcal{D}([0,T]; \R^d)\bigg)^{N} \times \bigg(\mathcal{D}([0,T]; \R^d)\bigg)^{N}\rightarrow \bigg(\mathcal{C}([0,T]; \R^d)\bigg)^{N}
$$
\noindent such that
$$
(X^{1}, \cdots, X^{N}) = \Phi_N((\xi^1, \cdots, \xi^{N}), (W^1, \cdots, W^{N}), (\widetilde{Z}^1,\cdots, \widetilde{Z}^{N}), (\widehat{Z}^1,\cdots, \widehat{Z}^{N}))
$$

\noindent Moreover, by symmetry of the dynamics \eqref{euler:approximation:scheme}, for any permutation $\zeta $ of $\left\{1, \cdots, N\right\}$, we get
$$
(X^{\zeta(1)}, \cdots, X^{\zeta(N)}) = \Phi_N((\xi^{\zeta(1)}, \cdots, \xi^{\zeta(N)}), (W^{\zeta(1)}, \cdots, W^{\zeta(N)}), (\widetilde{Z}^{\zeta(1)},\cdots, \widetilde{Z}^{\zeta(N)}), (\widehat{Z}^{\zeta(1)},\cdots, \widehat{Z}^{\zeta(N)})).
$$

Now, combining the fact that $(\widehat{Z}^1, \cdots, \widehat{Z}^{N})$ and $\big((\xi^1, \cdots, \xi^{N}), (W^1, \cdots, W^{N}), (\widetilde{Z}^1,\cdots, \widetilde{Z}^{N})\big)$ are independent together with the fact that $(\xi^{i}, W^i, \widetilde{Z}^i)_{1\leq i \leq N}$ are i.i.d, we deduce that conditionally on $(\widehat{Z}^1, \cdots, \widehat{Z}^{N})$ the processes $(X^{i})_{1\leq i \leq N}$ are identically distributed.
\vskip2pt
\noindent \emph{Step 2:} The integral against $\widehat{Z}^{i}$ generates jumps at discrete instants, that is, one may write the restriction of $\widetilde{N}^i$ to the set $[0,\infty) \times \left\{ z \geq 1 \right\}$ as $\sum_{n\geq1}\delta_{(T^{i}_n, Z_n^i)}$ where $(T^{i}_n)_{n\geq1}$ are the jump times of a Poisson process $J^{i}$ with intensity $\lambda = \int_{1}^{\infty} \nu(dz)$, the random variables $(Z^{i}_n)_{n\geq1}$ being i.i.d. with law $\lambda^{-1} \textbf{1}_{z\geq1} \nu(dz)$. We denote by $\G^{i}:=\sigma(T^{i}_n, n\geq1)$ the $\sigma$-algebra generated by all the jump times of the compound Poisson process $\widehat Z^{i}$.

We now prove some conditional moment estimate for $X^{i}$. As the computations are similar between two successive instants $T^{i}_{n}, \, T^{i}_{n+1}$, we only give the estimate on the interval $[T^{i}_1\wedge T, T^{i}_2\wedge T]$. The dynamics after $T^{i}_1\wedge T$ and strictly before $T^{i}_2\wedge T$ is given by
\begin{align}
X^{i}_t & = X^{i}_{T^{i}_1\wedge T} + \int_{(T^{i}_1\wedge T,t]} \widetilde b(X^{i}_{\eta(s)}, \mu_{{\eta(s)}}) ds + \int_{(T^{i}_1\wedge T,t]} \sigma(X^{i}_{\eta(s)}) dW^{i}_s + \int_{(T^{i}_1\wedge T,t]} h(X^{i}_{\eta(s)}) d\widetilde{Z}^{i}_s\label{eT1T2}
\end{align}
\noindent with $\widetilde{b}(x, \mu) := b(x, \mu) - h(x) \int_{z\geq1} z \nu(dz)$.

By using Lemma \ref{euler:moment} and the linear growth condition on $b$ and $h$, for any $t\in (T^{i}_1\wedge T, T)$, the absolute value of the drift can bounded as follows
\begin{align*}
& \Big|\int_{(T^{i}_1\wedge T,t]} \widetilde b(X^{i}_{\eta(s)}, \mu_{{\eta(s)}}) ds\,\Big| \\
& \leq  C_T \int^{t}_{T^{i}_1\wedge T} (1+ |X^{i}_{\eta(T^{i}_1\wedge T)}|)\, \1_{\{s \leq  \tau(T^{i}_1\wedge T)\}}  ds + C_T \int^{t}_{T^{i}_1\wedge T} (1+ |X^{i}_{\eta(s)}|)\, \1_{\{s >\tau(T^{i}_1\wedge T)\}}  ds + C_T \int^{t}_{T^{i}_1\wedge T} \frac{1}{N}\sum_{j=1}^{N}|X^{j}_{\eta(s)}|  \, ds \\
& \leq   C_T(1+ |X^{i}_{\eta(T^{i}_1\wedge T)}|)Tn^{-1} + C_T \int^{t}_{T^i_1\wedge T} \1_{\{\eta(s) \geq T^{i}_1 \wedge T \}} |X^i_{\eta(s)}| ds + C_T \int^{t}_{T^i_1\wedge T} \frac{1}{N}\sum_{j=1}^{N}|X^{j}_{\eta(s)}|  \, ds \\
& \leq C_T\left(1+ |X^i_{\eta(T^i_1\wedge T)}| + \frac{1}{N}\sum_{j=1}^{N}|X^{j}_{\eta(T^i_1\wedge T)}|  + \int^{t}_{T^i_1\wedge T} \1_{\{\eta(s) \geq T^i_1\wedge T\}}|X^i_{\eta(s))}| ds +  \int^{t}_{T^i_1\wedge T} \1_{\{\eta(s) \geq T^i_1\wedge T\}} \frac{1}{N}\sum_{j=1}^{N}|X^{j}_{\eta(s)}|  \, ds\right).
\end{align*}
where we have used the fact that $|\tau(s) - s| \leq Tn^{-1}$ for all $s\in [0,T]$.

 On the other hand, again by using Lemma \ref{euler:moment}, the stochastic integrals against the $L^2$-martingales $W^i$ and $\widetilde Z^i$ can be similarly decomposed into
\begin{align*}
\int_{(T^{i}_1\wedge T, t]} \sigma(X^i_{\eta(s)}) dW^i_s & = \sigma(X^i_{\eta(T^{i}_1\wedge T)})(W^i_{\tau(T^{i}_1\wedge T)} - W^i_{T^{i}_1\wedge T}) + \int_{(T^{i}_1\wedge T, t]} \sigma(X^i_{\eta(s)})\1_{\{\eta(s) \geq T^{i}_1\wedge T \}} dW^i_s.
\end{align*}
 By using the Burkholder-Davis-Gundy inequality and then the Jensen inequality, we obtain
\begin{align}
 \mathbb{E}\big[\sup_{T^i_1 \wedge T \leq t < T^i_2 \wedge T } |X^i_t|^\beta \, \big| \,\F_{T^i_1 \wedge T} \vee \sigma(\widehat{Z}^1, & \cdots, \widehat{Z}^{N})\big] 
\leq C_T\Bigg(1+ |X^i_{T^i_1\wedge T}|^\beta\vee |X^i_{\eta(T^i_1\wedge T)}|^\beta  + \frac{1}{N}\sum^N_{j=1} |X^j_{\eta(T^i_1\wedge T)}|^\beta \nonumber  \\
& \quad + \int^{T^i_2\wedge T}_{T^i_1\wedge T}  \1_{\{\eta(s) \geq T^i_1\wedge T\}} \mathbb{E}\big[|X^i_{\eta(s)}|^\beta\big| \,\F_{T^i_1 \wedge T} \vee \sigma(\widehat{Z}^1, \cdots, \widehat{Z}^{N})\big] ds \nonumber \\
& \quad + \int^{T^i_2\wedge T}_{T^i_1\wedge T}  \1_{\{\eta(s) \geq T^i_1\wedge T\}} \frac{1}{N}\sum_{j=1}^{N} \mathbb{E}\big[|X^j_{\eta(s)}|^\beta\big|\, \,\F_{T^i_1 \wedge T} \vee \sigma(\widehat{Z}^1, \cdots, \widehat{Z}^{N})\big] \,  ds \nonumber \\
& \quad +  \mathbb{E}\bigg[\,\bigg[\int^{T^i_2\wedge T}_{T^i_1\wedge T}   |X^i_{\eta(s)\vee (T^i_1 \wedge T)}|^2 ds \bigg]^\frac{\beta}{2}\, \bigg| \,\F_{T^i_1\wedge T} \vee \sigma(\widehat{Z}^1, \cdots, \widehat{Z}^{N})\bigg]\Bigg)\label{tempor:inequality:before:Gronwall} . 
\end{align}
Next, in order to apply Gr\"onwall's inequality, we first notice that 
\begin{align*}
\left[\int^{T^{i}_2\wedge T}_{T^{i}_1\wedge T}   |X^{i}_{\eta(s)\vee (T^{i}_1 \wedge T)}|^2 ds \right]^\frac{\beta}{2} 
& = \left[\int^{T^{i}_2\wedge T}_{T^{i}_1 \wedge T}   \left(|X^{i}_{\eta(s)\vee (T^{i}_1\wedge T)}|^{\beta}\right)^{\frac{2}{\beta}(1-\frac{\beta}{2})}|X^i_{\eta(s)\vee (T^{i}_1\wedge T)}|^\beta ds \right]^\frac{\beta}{2} \\
& \leq  \left[\int^{T^{i}_2\wedge T}_{T^{i}_1\wedge T}   \left(\sup_{T^{i}_1\wedge T \leq s < T^{i}_2\wedge T} |X^i_{\eta(s)\vee (T^{i}_1\wedge T)}|^{\beta}\right)^{\frac{2}{\beta}(1-\frac{\beta}{2})}|X^i_{\eta(s)\vee (T^{i}_1\wedge T)}|^\beta ds \right]^\frac{\beta}{2} \\
& \leq   \left[\frac{1}{4C_T}\sup_{T^{i}_1\wedge T\leq s < T^{i}_2\wedge T} |X^{i}_{\eta(s)\vee (T^{i}_1\wedge T)}|^{\beta}\right]^{(1-\frac{\beta}{2})} \left[(4C_T)^{(1-\frac{\beta}{2})\frac{2}{\beta}}\int^{T^{i}_2\wedge T}_{T^{i}_1 \wedge T}  |X^{i}_{\eta(s)\vee (T^{i}_1\wedge T)}|^\beta ds \right]^\frac{\beta}{2} 
\end{align*}
and then we apply the Young inequality with $p^{-1} = 1-\frac{\beta}{2}$ and $q^{-1} = \frac{\beta}{2}$ to obtain
\begin{align*}
\left[\int^{T^i_2\wedge T}_{T^i_1\wedge T}   |X^i_{\eta(s)\vee (T^i_1\wedge T)}|^2 ds \right]^\frac{\beta}{2} \leq \frac{1}{4C_T}\sup_{T^i_1\wedge T \leq s < T^i_2\wedge T} |X^i_{\eta(s)\vee (T^i_1\wedge T)}|^{\beta} + C_{T,\beta}\int^{T^i_2\wedge T}_{T^i_1 \wedge T}  |X^{i}_{\eta(s)\vee (T^i_1\wedge T)}|^\beta ds .
\end{align*}
Note that since $C_T \times \frac{1}{4C_T} = \frac14  <1$, we can move the first term appearing on the right-hand side of the above inequality to the left hand side of \eqref{tempor:inequality:before:Gronwall} and obtain
\begin{align*}
& \mathbb{E}\big[\sup_{T^i_1 \wedge T \leq t < T^i_2 \wedge T } |X^i_t|^\beta \,\big|\,\F_{T^i_1 \wedge T} \vee \sigma(\widehat{Z}^1, \cdots, \widehat{Z}^{N})\big]\\
& \leq C_T\left(1+ \mathbb{E}[\sup_{0\leq t \leq T^i_1\wedge T} |X^i_{t}|^\beta|\,\F_{T^i_1 \wedge T} \vee \sigma(\widehat{Z}^1, \cdots, \widehat{Z}^{N})] + \frac{1}{N}\sum^N_{j=1}  \mathbb{E}[|X^j_{\eta(T^i_1\wedge T)}|^\beta | \,\F_{T^i_1 \wedge T} \vee \sigma(\widehat{Z}^1, \cdots, \widehat{Z}^{N})] \right) \\
& \quad + C_T \int^{T^i_2\wedge T}_{T^i_1\wedge T}  \1_{\{\eta(s) \geq T^i_1\wedge T\}} \mathbb{E}\big[|X^i_{\eta(s)}|^\beta\big|\, \F_{T^i_1 \wedge T} \vee \sigma(\widehat{Z}^1, \cdots, \widehat{Z}^{N})\big] \, ds \\
& \quad + C_T \int^{T^i_2\wedge T}_{T^i_1\wedge T}  \1_{\{\eta(s) \geq T^i_1\wedge T\}} \frac{1}{N}\sum_{j=1}^{N} \mathbb{E}\big[|X^j_{\eta(s)}|^\beta\big|\,\F_{T^i_1 \wedge T} \vee \sigma(\widehat{Z}^1, \cdots, \widehat{Z}^{N})\big]   \, ds.
\end{align*}

In order to deal with the average term, we take the conditional expectation with respect to $\sigma(\widehat{Z}^1, \cdots, \widehat{Z}^{N})$ in the preceding inequality, from the tower property of conditional expectation and the conclusion of step 1, we obtain 
$$
\frac{1}{N}\sum_{j=1}^N \mathbb{E}\big[|X^j_{\eta(s)}|^\beta\big|\,\sigma(\widehat{Z}^1, \cdots, \widehat{Z}^{N})\big] = \mathbb{E}\big[|X^i_{\eta(s)}|^\beta\big|\,\sigma(\widehat{Z}^1, \cdots, \widehat{Z}^{N})\big]. $$
We point out here that since $X^i_{\eta(s)}$ is the Euler-Maruyama scheme at a grid point, it is clear that it is a functional of the Brownian and the L\'evy increments. From the previous computations, we thus get
\begin{align*}
& \mathbb{E}\big[\sup_{T^i_1 \wedge T \leq t < T^i_2 \wedge T } |X^i_t|^\beta \,\big|\, \sigma(\widehat{Z}^1, \cdots, \widehat{Z}^{N})\big]\\
& \leq C_T\left(1+ \mathbb{E}[\sup_{0\leq t \leq T^i_1\wedge T} |X^i_{t}|^\beta|\, \sigma(\widehat{Z}^1, \cdots, \widehat{Z}^{N})] + \int^{T^i_2\wedge T}_{T^i_1\wedge T}  \mathbb{E}\big[ \sup_{T^i_1\wedge T \leq t < s\wedge (T^i_2\wedge T)}|X^i_{t}|^\beta\big|\, \, \sigma(\widehat{Z}^1, \cdots, \widehat{Z}^{N})\big]\right) 
\end{align*}

\noindent so that, by Gr\"onwall's inequality
\begin{align*}
 \mathbb{E}\big[\sup_{T^i_1 \wedge T \leq t < T^i_2 \wedge T} |X^i_t|^\beta \,\big|\, \sigma(\widehat{Z}^1, \cdots, \widehat{Z}^{N}) \big]& \leq C_T\Big(1+ \mathbb{E}[\sup_{0\leq t \leq T^i_1\wedge T} |X^i_{t}|^\beta| \sigma(\widehat{Z}^1, \cdots, \widehat{Z}^{N})] \Big),
\end{align*}
and, taking conditional expectation w.r.t $\mathcal{G}^{i}$, we thus obtain
\begin{align}
 \mathbb{E}\big[\sup_{T^i_1 \wedge T \leq t < T^i_2 \wedge T} |X^i_t|^\beta \,\big|\, \G^i \big]& \leq C_T\Big(1+ \mathbb{E}[\sup_{0\leq t \leq T^i_1\wedge T} |X^i_{t}|^\beta| \G^i] \Big). \label{first:step:beta:moment:euler:scheme}
\end{align}
In particular, the preceding upper-bound yields
\begin{align}
\mathbb{E}\big[ \sup_{T^{i}_1 \wedge T \leq t \leq T^{i}_2 \wedge T}  |X^{i}_t|^{\beta}\big| \,\G^i \big] 
& \leq C_T\Big(1+ \mathbb{E}[\sup_{0\leq t \leq T^i_1\wedge T} |X^i_{t}|^\beta| \,\G^i ]  \Big) + \mathbb{E}\big[|X^i_{T^{i}_2}|^{\beta}\big|\,\G^i \big]  \textbf{1}_{\left\{T^i_2< T\right\}}\label{rec1}.
\end{align}
\noindent Now at the jump time $T^i_2$, $X^{i}_{T^{i}_2} = X^{i}_{T^{i}_2-}+h(X^i_{\eta(T^i_2)}) Z^i_2$ and by using the linear growth assumption of coefficient $h$ (uniformly on $[0,T]$), we have $|X^{i}_{T^{i}_2}|^\beta \leq C_T(1+|X^{i}_{T^{i}_2-}|^\beta \vee |X^{i}_{\eta(T^{i}_2)}|^\beta)(1+|Z^{i}_2|^\beta)$ on the set $\left\{ T^{i}_2 < T\right\}$. From \eqref{first:step:beta:moment:euler:scheme} and the fact that $\int_{1}^{\infty} z^\beta \nu(dz) < \infty$, we obtain
\begin{align}
\mathbb{E}[|X^{i}_{T^i_2}|^\beta | \G^i] & \leq C_T(1+ \mathbb{E}[|Z_2|^\beta])\Big(1+ \mathbb{E}[\sup_{0\leq t <  T^i_2\wedge T} |X^i_{t}|^\beta|\G^i]  \Big)\nonumber \\
										& \leq C_T(1+ \mathbb{E}[|Z_2|^\beta])\Big(1+ \mathbb{E}[\sup_{0\leq t \leq  T^i_1\wedge T} |X^i_{t}|^\beta|\G^i]  + \mathbb{E}[\sup_{T^i_1\wedge T \leq t <  T^i_2\wedge T} |X^i_{t}|^\beta|\G^i]\Big).\nonumber \\
										& \leq C_T(1+ \mathbb{E}[|Z_2|^\beta])\Big(1+ \mathbb{E}[\sup_{0\leq t \leq  T^i_1\wedge T} |X^i_{t}|^\beta|\G^i] \Big).\label{rec2}
\end{align}
From \eqref{rec1} and \eqref{rec2} we deduce that there exists constant $M_{T}$ (now depending on $\int_{1}^{\infty} z^\beta \nu(dz)$) such that 
\begin{align*}
\mathbb{E}\big[\sup_{T^{i}_1 \wedge T \leq t \leq T^{i}_2 \wedge T}  |X^{i}_t|^{\beta}\big|\,\G^i \big] & \leq M_T\Big(1+ \mathbb{E}[\sup_{0\leq t \leq T^i_1\wedge T} |X^i_{t}|^\beta|\G^i]  \Big).
\end{align*}

Performing similar computations on any time interval $[T^{i}_{n}\wedge T, T^{i}_{n+1}\wedge T]$, one deduces that
\begin{align}
\mathbb{E}\big[\sup_{T^{i}_{n} \wedge T \leq t \leq T^{i}_{n+1} \wedge T}  |X^{i}_t|^{\beta}\big|\, \G^i\big] & \leq M_T\Big(1+ \mathbb{E}[\sup_{0\leq t \leq T^i_n\wedge T} |X^i_{t}|^\beta|\,\G^i]  \Big)\label{rec3}
\end{align}
for any integer $n$, with the convention $T^{i}_0= 0$. 

Thus for any pair $T^i_j < T^i_k$, by setting $S_{[T^i_j, T^i_k]}:= \mathbb{E}[\sup_{T^i_j\wedge T\leq s \leq T^i_k\wedge T}|X^i_s|^\beta\,|\,\G]$ and by using \eqref{rec3}, we observe that
\begin{align*}
1+ S_{[0,T^i_{k+1}]}  \leq 1+ S_{[0,T^i_{k}]} + S_{[T^i_k,T^i_{k+1}]} \leq 1+ S_{[0,T^i_{k}]} +M_T(1+ S_{[0,T^i_{k}]} ) \leq (1+M_T)(1+ S_{[0,T^i_{k}]}),
\end{align*}
which implies that $1+ S_{[0,T^i_{k+1}]} \leq (1+ M_\beta(\mu))(1+M_T)^{k+1}$ where $M_\beta(\mu) = \int_{\R} |x|^\beta d\mu(x)$. Finally by setting $K_T := (1+ M_\beta(\mu))(1+M_T)$, we obtain
\begin{align*}
\mathbb{E}[\sup_{0\leq t \leq T}|X_t^i|^\beta] & = \sum_{n\geq0} \mathbb{E}[\sup_{0\leq t \leq T}|X^i_t|^\beta \textbf{1}_{\left\{J^{i}_T=n\right\}}] \\
& =  \sum_{n\geq0} \mathbb{E}\big[\textbf{1}_{\left\{T^{i}_n\leq T < T^{i}_{n+1}\right\}} \mathbb{E}\big[\sup_{0\leq t \leq T^{i}_{n+1}\wedge T}|X^i_t|^\beta |\mathcal{G}^{i}\big] \big]\\
& \leq \sum_{n\geq0} K^{n+1}_T \frac{(\lambda T)^{n}}{n!} e^{-\lambda T} \,ds < \infty.
\end{align*}

\noindent \emph{Step 3:} We now prove the moment estimate on the time increments of the Euler-Maruyama scheme. We note that from the linear growth assumption of the coefficients and Jensen's inequality for any $t\in [0,T]$ it holds
	\begin{align}\label{inequality:moment:estimate:time:increment:euler:scheme}
 	|{X}_{t}^{n,i,N}-{X}_{\eta(t)}^{n,i,N}|^\beta
& \leq
	K\Big[1+|{X}_{\eta(t)}^{n,i,N}|^\beta + \frac{1}{N}\sum_{j=1}^N |X^{n, j,N}_{\eta(t)}|^\beta	 \Big]
	\left(|t-\eta(t)|^\beta
	+ |W^{i}_{t}-W^{i}_{\eta(t)}|^\beta+|Z^{i}_{t}-Z^{i}_{\eta(t)}|^\beta\right).
	\end{align}
We observe that $\max_{1\leq i \leq N}\e[|W^{i}_{t} - W^{i}_{\eta(t)}|^\beta] \leq C n^{-\frac{\beta}{2}}$ and
\begin{align*}
\max_{1\leq i \leq N}\e[|Z^{i}_t - Z^{i}_{\eta(t)}|^{\beta}] & \leq C \Big( (t-\eta(t))^{\frac{\beta}{2}} + (t-\eta(t))^{\beta-1} \mathbb{E}\bigg[\int_{\eta(t)}^{t} \int_{z\geq1} z^\beta N(ds, dz)\bigg] \bigg) \\
	& \leq C \Big(n^{-\frac{\beta}{2}} + n^{-\beta} \int_{z\geq1} z^\beta \nu(dz) \Big) \leq C n^{-\frac{\beta}{2}}.
\end{align*}
Then, taking the expectation in \eqref{inequality:moment:estimate:time:increment:euler:scheme}, using the first moment estimate established in the previous step together with the fact that the processes $(X^{n, i, N})_{1\leq i \leq N}$, are identically distributed, we obtain
\begin{align*}
\max_{1\leq i \leq N}\sup_{0\leq t \leq T}\mathbb{E}[|{X}_{t}^{n,i,N}-{X}_{\eta(t)}^{n,i,N}|^\beta]
& \leq
	C \Big[1+\max_{1\leq i \leq N}\mathbb{E}[\,\sup_{0\leq t \leq T}|{X}_{t}^{n,i,N}|^\beta]\Big]
	\Big(n^{-\beta}
	+ n^{-\frac{\beta}{2}}\Big) \leq C n^{-\frac{\beta}{2}}.
	\end{align*}
\end{proof}

\begin{Thm}\label{euler:main}
Suppose that the assumptions of Theorem \ref{strong:existence:uniqueness} and assumption \ref{Ass_2} hold. Then, for all $T>0$, there exists a positive constant $C$ (independent of $n$ and $N$) such that
	\begin{align*}
	\max_{1\leq i \leq N}\sup_{0\leq t \leq T}\e[|X^{i,N}_{t}-{X}_{t}^{n,i,N}|]
	\leq C \varepsilon_n
	\end{align*}
\noindent with
$$
\varepsilon_n:=
\left\{ \begin{array}{ll}
	\displaystyle
	n^{-\frac12 \rho\wedge \eta}
	+ n^{- p^{*} },
	& \textnormal{ if } \gamma \in (1/2,\beta_\nu /2),\\
	\displaystyle
	\log(n)^{-1},
	& \textnormal{ if } \gamma=1/2\\
	\end{array}\right.
$$
and
$$
p^{*} := \left\{ \begin{array}{ll}
\displaystyle
\gamma\wedge \eta-\frac12\frac{\gamma\wedge \eta}{\gamma},& \textnormal{ if } \alpha_\nu \in [1, 2 \frac{(1-\gamma)}{1-\eta}], \\
 \eta-\frac{\eta}{2-\alpha_\nu (1-\eta)-\delta}  ,& \textnormal{ if } \alpha_\nu \in(2 \frac{(1-\gamma)}{1-\eta},2]
\end{array}\right.
$$
\noindent for any $\delta \in (0, 1-(1-\eta)\alpha_\nu)$. In the special case where $\nu(dz)$ is an $\alpha$-stable like L\'evy measure for $\alpha \in [1,2]$, we can take $\alpha_\nu =\alpha$ and $\delta = 0$ in $p^*$.

\end{Thm}

\begin{Rem}
Before proceeding to the proof, let us make some comments about the convergence rate obtain in Theorem \ref{euler:main}. First, observe that in the case $\gamma \leq \eta$ (or equivalently $2(1-\gamma)/(1-\eta)\geq 2$), i.e. when the jump coefficient $h(.)$ is more regular than the diffusion coefficient $\sigma(.)$, one obtains a convergence rate of order $n^{-\frac12 \rho\wedge \eta} + n^{-(\gamma-\frac12)}$ for any $\alpha_\nu \in [1,2]$ which is the same as the one established in the Brownian setting by Gy\"ongy and R\'asonyi \cite{Gyongy:rasonyi}. Otherwise, if $\gamma>\eta$, the rate is worse than the one corresponding to the Brownian setting. Moreover, it becomes poorer and poorer as $\alpha_{\nu}$ increases, starting from $n^{-\frac12 \rho\wedge \eta} + n^{- (\eta - \frac{\eta}{2\gamma})}$, which is less than the rate of the Brownian setting, if $\alpha_\nu \in [1, 2 (1-\gamma)/(1-\eta)]$ down to $n^{-\frac12 \rho\wedge \eta} + n^{-(\eta-\frac{1}{2-\delta})} \approx n^{-\frac12 \rho\wedge \eta} + n^{-(\eta- \frac12)}$, for any $\delta \in (0,2\eta-1)$, if $\alpha_\nu=2$. We emphasize that the latter rate of convergence should not come as a big surprise since, in the case of stable-like L\'evy measure, it formally corresponds to the rate of the Brownian setting with a jump coefficient $h$ being $\eta$-H\"older continuous, $\eta\in (1/2, 1]$.
\end{Rem}
\vskip 4pt

\begin{proof}
For $\varepsilon \in (0,1)$ and $\delta \in (0,1)$, to be chosen later, we apply the Yamada-Watanabe function $\phi_{\delta, \varepsilon}$ to the difference $Y^{n, i} := X^{i,N} - X^{n, i, N}$. Employing \eqref{phi3} and then It\^o's formula, we get
\begin{equation}
\label{ito:formula:strong:convergence:rate}
|Y^{n, i}_t| \leq \varepsilon+ \phi_{\delta, \varepsilon}(Y^{n, i}_t) = \varepsilon + {M}_t^{i, n, \delta,\varepsilon} +{I}_t^{i, n, \delta,\varepsilon} +{J}_t^{i, n, \delta,\varepsilon} +{K}_t^{i, n, \delta,\varepsilon},
\end{equation}

 \noindent where we set
\begin{align*}
	{M}_t^{i, n, \delta,\varepsilon}
	:=& \int_{0}^{t} \phi_{\delta,\varepsilon}'(Y^{n,i}_{s}) \, (\sigma(X^{i, N}_{s}) - \sigma(X^{n, i, N}_{\eta(s)})) \, dW^{i}_{s} \\
	& \quad +\int_{0}^{t} \int_{0}^{\infty}\left\{\phi_{\delta,\varepsilon}(Y^{n, i}_{s-} + (h(X^{i,N}_{s-})-h(X^{n, i, N}_{\eta(s)}))z)-\phi_{\delta,\varepsilon}(Y^{n,i}_{s-})\right\}
	\widetilde{N}^{i}(ds,dz),\\
	{I}_t^{i, n, \delta,\varepsilon}
	:=& \int_{0}^{t} \phi_{\delta,\varepsilon}'(Y^{n,i}_s) (b(X^{i,N}_{s},\mu^{N}_s)-b(X_{\eta(s)}^{n, i, N},\mu^{n,N}_{\eta(s)}))ds,
	\quad\\
	{J}_t^{i, n, \delta,\varepsilon}
	:=& \frac{1}{2} \int_{0}^{t} \phi_{\delta,\varepsilon}''(Y^{n, i}_{s}) |\sigma(X^{i, N}_{s}) - \sigma(X^{n, i, N}_{\eta(s)})|^2ds,\\
	{K}_t^{i,n, \delta,\varepsilon}
	:=&
	\int_{0}^{t} \int_{0}^{\infty}
	\Big\{
	\phi_{\delta,\varepsilon}(Y^{n, i}_{s-} + (h(X^{i,N}_{s-})-h(X^{n, i, N}_{\eta(s)}))z) - \phi_{\delta,\varepsilon}(Y^{n,i}_{s-}) -  (h(X^{i,N}_{s-})-h(X^{n, i, N}_{\eta(s)}))z \phi_{\delta,\varepsilon}'(Y^{n,i}_{s-})
	\Big\}
	\nu(dz)ds.
\end{align*}

Once again, similarly to the proof of Theorem \ref{strong:existence:uniqueness}, one needs to employ a localization technique using a sequence $(\tau_m)_{m\geq1}$ so that $(M^{i, n, \delta, \varepsilon}_{t\wedge \tau_m})_{t\geq0}$ is an $L^{1}(\mathbb{P})$-martingale and then pass to the limit as $m\uparrow \infty$ using the fact that $Y^{n, i}$ is right-continuous and Fatou's lemma. Since this procedure is standard, we omit it for sake of simplicity. We now quantify the contribution of the terms ${I}_t^{i, n, \delta,\varepsilon}, \, {J}_t^{i, n, \delta,\varepsilon}$ and ${K}_t^{i,n, \delta,\varepsilon}$. For the rest of the proof, we denote by $C$ a positive constant that may change from line to line but depends neither on $n$ nor $N$. 

For the first term associated to the difference of the drift, we make use of assumption \ref{Ass_1} (ii)
	\begin{align*}
	I_t^{i, n,\delta,\varepsilon}
	& = \int_{0}^{t} \phi_{\delta,\varepsilon}'(Y^{n,i}_s) (b(X^{i,N}_{s},\mu^{N}_s)-b(X_{\eta(s)}^{n,i,N},\mu^{n,N}_{\eta(s)}))ds\\
	& =  \int_{0}^{t} \phi_{\delta,\varepsilon}'(Y^{n,i}_s) (b(X_{s}^{i,N},\mu^N_s)- b(X^{n,i,N}_{s},\mu^{n,N}_{\eta(s)}))ds  + \int_{0}^{t} \phi_{\delta,\varepsilon}'(Y^{n,i}_s) (b(X_{s}^{n, i, N},\mu^{n, N}_{\eta(s)})- b(X_{\eta(s)}^{n, i, N},\mu^{n,N}_{\eta(s)}))ds \\
	&\leq [b]_L \int_{0}^{t} |Y^{n,i}_s| ds + [b]_L  \int^t_0 W_1(\mu^N_s, \mu^{n,N}_{\eta(s)})ds + \kappa \int_{0}^{t} |X^{n, i, N}_s - X^{n, i, N}_{\eta(s)}|^\rho ds .\nonumber
	\end{align*}
	
From the very definition of the Wasserstein distance and the triangle inequality
\begin{align*}
W_1(\mu^N_s, \mu^{n,N}_{\eta(s)})  \leq \frac{1}{N}\sum_{j=1}^{N} |Y^{n, j}_s| + |X^{n, j, N}_s - X^{n, j,N}_{\eta(s)}|
\end{align*}

\noindent so that
\begin{align}
I_t^{i, n, \delta, \varepsilon}
	&\leq C\left\{ \int_{0}^{t} |Y^{n,i}_s| ds+ \int_0^t \frac{1}{N}\sum_{j=1}^{N} |Y^{n, j}_s| \, ds  + \int_{0}^{t} \left[ |X^{n, i, N}_s - X^{n, i, N}_{\eta(s)}| + |X^{n, i, N}_s - X^{n, i, N}_{\eta(s)}|^\rho \right] ds\right\}. \label{first:bound:drift:term}
\end{align}


	To estimate $J^{i, n,\delta, \varepsilon}_t$, we make use of assumption 2.1 (iii) and \eqref{phi4} so that 
	\begin{align}
	J_t^{i, n,\delta,\varepsilon}
	&\leq 2 \int_{0}^{t}  \phi_{\delta,\varepsilon}''(Y^{n,i}_s) |\sigma(X_{s}^{i,N})-\sigma(X^{n,i,N}_{s})|^2ds
	+2 \int_{0}^{t} \phi_{\delta,\varepsilon}''(Y^{n,i}_s) |\sigma(X_{s}^{n,i,N})-\sigma(X^{n,i,N}_{\eta(s)})|^2 ds \notag\\
	&\leq C \bigg[\int_{0}^{t} \frac{\1_{[\varepsilon/\delta,\varepsilon]}(|Y^{n,i}_s|) |Y^{n,i}_s|^{2\gamma}}{|Y^{n,i}_s|\log \delta}ds
	+  \int_{0}^{t} \frac{\1_{[\varepsilon/\delta,\varepsilon]}(|Y^{n,i}_s|) |X^{n,i,N}_{s}-X^{n,i,N}_{\eta(s)}|^{2\gamma}}{|Y^{n,i}_s|\log \delta}ds\bigg] \notag \\
	&\leq C \bigg[\frac{\varepsilon^{2\gamma-1}}{\log \delta}
	+ \frac{\delta}{\varepsilon \log \delta}  \int_{0}^{t} |X^{n,i,N}_{s}-X^{n,i,N}_{\eta(s)}|^{2\gamma}ds\bigg].\label{EP_7}
	\end{align}	
	
	We then decompose $K_t^{i, n,\delta,\varepsilon}$ as the sum of two terms, namely
	\begin{align*}
	K_t^{i, n,\delta,\varepsilon}
	=K_t^{i, n,\delta,\varepsilon,1}+K_t^{i, n,\delta,\varepsilon,2}
	\end{align*}
	where $K_t^{i, n,\delta,\varepsilon,1}$ and $K_t^{i, n,\delta,\varepsilon,2}$ are given by
	\begin{align*}
	K_t^{i, n,\delta,\varepsilon,1}
	&:=
	\int_{0}^{t} \int_{0}^{\infty} \textbf{1}_{\left\{ Y^{n, i}_s \neq 0 \right\}}
	\Big\{
	\phi_{\delta,\varepsilon}(Y^{n,i}_s+\{h(X_{s}^{i,N})-h(X_{s}^{n,i,N})\}z)-\phi_{\delta,\varepsilon}(Y_{s}^{n,j})\\
	& \quad  -\{h(X_{s}^{i,N})-h(X_{s}^{n,i,N})\}z \phi_{\delta,\varepsilon}'(Y^{n,i}_s)
	\Big\}
	\nu(dz)ds,\\
	K_t^{i, n,\delta,\varepsilon,2}
	&:=
	\int_{0}^{t} \int_{0}^{\infty}
	\Big\{
	\phi_{\delta,\varepsilon}(Y^{n,i}_s+\{h(X_{s}^{i,N})-h(X_{\eta(s)}^{n,i,N})\}z)-\phi_{\delta,\varepsilon}(Y^{n,i}_s+\{h(X_{s}^{i,N})-h(X_{s}^{n,i,N})\}z)\\
	& \quad -\{h(X_{s}^{n,i,N})-h(X_{\eta(s)}^{n,i,N})\}z \phi_{\delta,\varepsilon}'(Y^{n,i}_s)
	\Big\}
	\nu(dz)ds
	\end{align*}

	\noindent where for the first integral we used the fact if $Y^{n ,i}_s=0$ then $h(X_{s}^{i,N})-h(X_{s}^{n, i, N})=0$, for any $s \in [0,t]$. Now, let $y=Y^{n,i}_s$ and $x=h(X_{s}^{i,N})-h(X_{s}^{n,i,N})$. If $y x> 0$ we can apply Lemma \ref{key_lem_0}. For any $u>0$, we get 
	\begin{align}
	&\int_{0}^{\infty}
	\left\{
	\phi_{\delta,\varepsilon}(Y^{n,i}_s+\{h(X^{i, N}_{s})-h(X_{s}^{n,i ,N})\}z)-\phi_{\delta,\varepsilon}(Y^{n,i}_s)
	-\{h(X^{i, N}_{s})-h(X_{s}^{n, i, N})\}z \phi_{\delta,\varepsilon}(Y^{n,i}_s)
	\right\}
	\nu(dz) \notag\\
	&\leq
	\frac{2|h(X^{i, N}_{s})-h(X_{s}^{n, i, N})|^2 \1_{(0,\varepsilon]}(|Y^{n,i}_s|)}{|Y^{n,i}_s| \log \delta} \int_{0}^{u} z^2 \nu(dz)
	+ 2|h(X^{i, N}_{s})-h(X_{s}^{n, i, N})| \1_{(0,\varepsilon]}(|Y^{n,i}_s|) \int_{u}^{\infty} z \nu(dz)\notag\\
	& \leq C\bigg[\frac{ |Y^{n,i}_s|^{2 \eta} \1_{(0,\varepsilon]}(|Y^{n,i}_s|)}{|Y^{n,i}_s| \log \delta} \int_{0}^{u} z^2 \nu(dz)
	+  |Y^{n,i}_s|^{\eta} \1_{(0,\varepsilon]}(|Y^{n,i}_s|) \int_{u}^{\infty} z \nu(dz)\bigg] \notag \\
	& \leq C\bigg[  \frac{\varepsilon^{2\eta-1}}{\log \delta} \int_{0}^{u} z^2 \nu(dz)
	+ \varepsilon^{\eta} \int_{u}^{\infty} z \nu(dz)\bigg]\label{EP_2}
	\end{align}
	where in the second last inequality, we used the fact that $h$ is an $\eta$-H\"older continuous function.
Next, by picking $u = \log(\delta)\varepsilon^{1-\eta}$, the quantity appearing on the right-hand side of \eqref{EP_2} can be further bounded by
	\begin{align*}
	C \bigg[\frac{1}{\log(\delta)^{\alpha_1-1}} \varepsilon^{2 \eta-1-(1-\eta)({\alpha}-2)} I^{\eta, \alpha_1}_{\varepsilon, \delta}
	+ \frac{1}{\log(\delta)^{\alpha_1-1}} \varepsilon^{\eta-(1-\eta)({\alpha_1}-1)}J^{\eta, \alpha_1}_{\varepsilon, \delta}\bigg]
	\leq C \frac{\varepsilon^{1-{\alpha_1}(1-\eta)}}{\log(\delta)^{\alpha_1-1}}  \big[I^{\eta, \alpha_1}_{\varepsilon, \delta} + J^{\eta, \alpha_1}_{\varepsilon, \delta}\big] \notag,
	\end{align*}
	where, for sake of clarity, for any $\alpha >\alpha_\nu$ and any $\eta \in (1-\frac{1}{\alpha_\nu},1]$, we introduced the quantities:
\begin{align*}
I^{ \eta, \alpha}_{\varepsilon, \delta} := [\log(\delta) \varepsilon^{1-\eta}]^{\alpha-2}\int_{0}^{\log(\delta) \varepsilon ^{1-\eta}} z^2 \nu(dz)\quad \mathrm{and} \quad J^{\eta, \alpha}_{\varepsilon, \delta} := [\log(\delta) \varepsilon^{1-\eta}]^{\alpha-1}\int_{\log(\delta) \varepsilon^{1-\eta}}^{\infty} z \nu(dz).
\end{align*}

We select $\alpha_1$ such that $\alpha_\nu< \alpha_1 < 1/(1-\eta)$. Observe that this choice is admissible since $\eta > 1-1/\alpha_\nu$.  Next, we note that from \eqref{left:right:tail:asymptotics}, one has $\sup_{\varepsilon \in (0,1)} \left\{I^{\eta, \alpha_1}_{\varepsilon, \delta} + J^{\eta, \alpha_1}_{\varepsilon, \delta} \right\}< \infty$ (if $\eta = 1$, we set $\delta = 2$ and bound the quantity $I^{\eta, \alpha_1}_{\varepsilon, \delta} + J^{\eta, \alpha_1}_{\varepsilon, \delta}$ by $2 \int_{0}^{\infty} z\wedge z^2 \nu(dz)$). 

%
On the other hand, if $y x < 0$, using assumption \ref{Ass_1} (iv) and following analogous computations as those used to derive \eqref{second:part:Kt:1}, we obtain
\begin{align*}
& \int_{0}^{\infty}\left\{ \phi_{\delta,\varepsilon}(Y^{n,i}_s+\{h(X^{i, N}_{s})-h(X_{s}^{n,i ,N})\}z)-\phi_{\delta,\varepsilon}(Y^{n,i}_s)
	-\{h(X^{i, N}_{s})-h(X_{s}^{n, i, N})\}z \phi_{\delta,\varepsilon}(Y^{n,i}_s)
	\right\}
	\nu(dz) \\
	& \leq  C  \left\{ \frac{\varepsilon}{\log(\delta)}  +  |Y^{n,i}_s| \right\}.
\end{align*}

From the above computations, we thus deduce that for any $\alpha_1 \in (\alpha_\nu, 1/(1-\eta))$
\begin{equation}
K_t^{i, n,\delta,\varepsilon,1} \leq C \bigg( \frac{\varepsilon}{\log(\delta)} + \frac{\varepsilon^{1-{\alpha_1}(1-\eta)}}{\log(\delta)^{\alpha_1-1}} +  \int^t_0 |Y^{n,i}_s| ds \bigg)\label{EP_51}
\end{equation}
\noindent for some positive constant $C$ independent of $n$ and $N$.

To estimate $K^{i, n,\delta, \varepsilon, 2}_t$, we apply Lemma \ref{key_lem12} with $y=Y^{n,i}_s,
		x=h(X_{s}^{i,N})-h(X_{\eta(s)}^{n,i,N})$ and $x'=h(X_{s}^{i,N})-h(X_{s}^{n,i,N})$.
Note that if $yx' < 0$, by assumption \ref{Ass_1} (iv), one has $0<-\sign(y)x' = |x'| \leq [h]_L |y|$. Therefore, for any $\alpha_2 \in (\alpha_\nu, 2]$, we obtain from Lemma \ref{key_lem12}
	\begin{align}
	K_t^{i, n,\delta,\varepsilon,2} 
	&\leq C\int^t_0\bigg\{|h(X_{s}^{n,i,N})-h(X_{\eta(s)}^{n,i,N})|^{\alpha_2} \frac{\delta}{\varepsilon \log \delta}  + |h(X_{s}^{n, i,N})-h(X_{\eta(s)}^{n,i,N})| \nonumber \\
	& \quad + \frac{1}{\log(\delta)}|h(X_{s}^{n,i,N})-h(X_{\eta(s)}^{n,i,N})| + |h(X_{s}^{n,i,N})-h(X_{\eta(s)}^{n,i,N})|	\nonumber \\
	& \quad + |h(X_{s}^{n,i,N})-h(X_{\eta(s)}^{n,i,N})| \bigg[\frac{1}{\varepsilon^{1-\eta} \log \delta} \int_0^{u} z^2 \nu(dz) + \int_{u}^{\infty} z\nu(dz)\bigg] \bigg\} ds \nonumber \\
	&\leq C \int^t_0\bigg\{|X_{s}^{n, i, N}-X_{\eta(s)}^{n, i, N}|^{\alpha_2\eta}  \frac{\delta}{\varepsilon \log \delta}  + |X_{s}^{n, i,N}- X_{\eta(s)}^{n, i, N}|^{\eta} + \frac{1}{\log(\delta)}|X_{s}^{n, i, N}-X_{\eta(s)}^{n, i, N}|^\eta \nonumber \\
	& \quad +  |X_{s}^{n, i, N}-X_{\eta(s)}^{n, i, N}|^\eta	 +  |X_{s}^{n,i,N}-X_{\eta(s)}^{n,i,N}|^\eta \bigg[\frac{1}{\varepsilon^{1-\eta} \log \delta} \int_0^u z^2 \nu(dz) + \int_u^{\infty} z \nu(dz)\bigg] \bigg\} ds. \label{EP_53}
\end{align}
\noindent 	By taking the expectation in \eqref{ito:formula:strong:convergence:rate}, \eqref{first:bound:drift:term}, \eqref{EP_7}, \eqref{EP_2}, \eqref{EP_51} and \eqref{EP_53}, for any $t\in [0,T]$, we obtain from Lemma \ref{EM_esti_1}, 
	\begin{align*}
	\e[|Y_t^{n,i}|]
	& \leq \varepsilon
	+\e[{I}_t^{n,\delta,\varepsilon}]
	+\e[{J}_t^{n,\delta,\varepsilon}]
	+\e[K_t^{n,\delta,\varepsilon}] \notag\\
	&\leq \varepsilon 
	+C \bigg[ \int_{0}^{t} \e[|Y_{s}^{n,i}|] ds + \frac{1}{n^{1/2}} + \frac{1}{n^{\rho/2}}
	+\frac{\varepsilon^{2\gamma-1}}{\log \delta} + \frac{\delta}{\varepsilon \log(\delta)}\frac{1}{n^{\gamma}}
	+ \frac{\varepsilon}{\log(\delta)}+ \frac{\varepsilon^{1-{\alpha_1}(1-\eta)} }{\log(\delta)^{\alpha_1-1}} \notag \\
	& \quad + \frac{1}{n^{\alpha_2 \eta/2}} \frac{\delta}{\varepsilon \log(\delta)} + \frac{1}{n^{\eta/2}} + \frac{1}{n^{\eta/2}}\frac{1}{\log(\delta)} + \frac{1}{n^{\eta/2}}\bigg[\frac{1}{\varepsilon^{1-\eta} \log \delta} \int_0^u z^2 \nu(dz) + \int_u^{\infty} z \nu(dz)\bigg] \bigg] \notag
	\end{align*}
where, due to integrability constraints, we require that $ 2\gamma< \beta_\nu $ and select $\alpha_2 \in (\alpha_\nu , \beta_\nu/\eta)\cap (\alpha_{\nu} , 2]$. Note that if $\alpha_\nu =2$ then one may take $\alpha_2 =2$. Then, 
	by using Gr\"onwall's inequality
	\begin{align}
 e^{-CT} \e[|Y_t^{n, i}|] 
 		&\leq
		\varepsilon
		+ C \bigg[ \frac{1}{n^{1/2}} + \frac{1}{n^{\rho/2}}
	+\frac{\varepsilon^{2\gamma-1}}{\log \delta} + \frac{\delta}{\varepsilon \log(\delta)}\frac{1}{n^{\gamma}}
	+ \frac{\varepsilon}{\log(\delta)}+ \frac{\varepsilon^{1-{\alpha_1}(1-\eta)} }{\log(\delta)^{\alpha_1-1}} \label{EP_13} \\
	& \quad + \frac{1}{n^{\alpha_2 \eta/2}} \frac{\delta}{\varepsilon \log(\delta)} + \frac{1}{n^{\eta/2}} + \frac{1}{n^{\eta/2}}\frac{1}{\log(\delta)} + \frac{1}{n^{\eta/2}}\bigg[\frac{1}{\varepsilon^{1-\eta} \log \delta} \int_0^u z^2 \nu(dz) + \int_u^{\infty} z \nu(dz)\bigg] \bigg]. \notag
	\end{align}
	To optimize the above bound, we proceed as follows.
	
	 Let us first consider the case $\gamma \in (1/2,\beta_\nu /2)$. In this case, we select $\delta =2$ and $u = \varepsilon^{1-\eta}$. We obtain
	\begin{align*}
	\e[|Y_t^{n, i}|]
	& \leq C
	\left\{
	\varepsilon
	+ \frac{1}{n^{1/2}} +\frac{1}{n^{\rho/2}}  
	+\varepsilon^{2\gamma-1} +\frac{1}{\varepsilon n^{\gamma}}
	+\varepsilon^{1-{\alpha_1}(1-\eta)}
	+\frac{1}{\varepsilon n^{\alpha_2\eta/2}} + \frac{1}{n^{\eta/2}} + \left(\frac{u^{2-\alpha_3}}{\varepsilon^{1-\eta}} + u^{1-\alpha_3}\right)\frac{1}{n^{\eta/2}} 
	\right\},\\
	& \leq C
	\left\{
	\varepsilon
	+ 	\frac{1}{n^{1/2}} +\frac{1}{n^{\rho/2}} 
	+\varepsilon^{2\gamma-1}+\frac{1}{\varepsilon n^{\gamma}}
	+\varepsilon^{1-{\alpha_1}(1-\eta)}
	+ \frac{1}{\varepsilon}\frac{1}{n^{\alpha_2\eta/2}} + \frac{1}{n^{\eta/2}} + \frac{\varepsilon^{(1-\eta)(2- \alpha_3)}}{\varepsilon^{1-\eta}}\frac{1}{n^{\eta/2}} 
	\right\}
	\end{align*}
\noindent where, for the first inequality, we used the fact that for $\varepsilon\geq 0$ sufficiently small, by \eqref{left:right:tail:asymptotics}, we have $\big[\varepsilon^{(1-\eta)(\alpha_3-2)} \int_0^{\varepsilon^{1-\eta}} z^2 \nu(dz) + \varepsilon^{(1-\eta)(1-\alpha_3)} \int_{\varepsilon^{1-\eta}}^{\infty} z \nu(dz)\big] < \infty$ for any $\alpha_3 > \alpha_\nu$.
We now regroup the above terms into
\begin{align*}
	\e[|Y_t^{n, i}|]	& \leq C
	\left\{\left[\frac{1}{n^{1/2}}
		+\frac{1}{n^{\rho/2}}+  \frac{1}{n^{\eta/2}}\right]
	+\left[\varepsilon+\varepsilon^{2\gamma-1}
	+\varepsilon^{1-{\alpha_1}(1-\eta)}\right]
	+\frac{1}{\varepsilon} \left[\frac{1}{n^{\gamma}} + \frac{1}{n^{\alpha_2\eta/2}}\right]  + \frac{1}{\varepsilon^{(\alpha_3-1)(1-\eta)}}\frac{1}{n^{\eta/2}} 
	\right\}\\
	& \leq C
	\left\{\frac{1}{n^{\zeta_1}}
	+ \frac{1}{n^{q\zeta_2}}
	+\frac{1}{n^{\zeta_3 - q}} + \frac{1}{n^{\zeta_4 - \zeta_5q}} 
		\right\}
	\end{align*}
where we have set $\varepsilon=n^{-q}$ and
\begin{align*}
\zeta_1 = \frac{1}{2}(\rho\wedge \eta), \quad \zeta_2 = (2\gamma-1) \wedge (1-\alpha_1(1-\eta)), \quad \zeta_3= \gamma \wedge \frac{\alpha_2\eta}{2}, \quad \zeta_4 = \frac{\eta}{2}, \quad \zeta_5 = (\alpha_3-1)(1-\eta)
\end{align*}

\noindent with the constraints: $\alpha_1 \in (\alpha_\nu, 1/(1-\eta))$, $\alpha_2 \in (\alpha_\nu, \beta_\nu/\eta) \cap (\alpha_\nu , 2]$ and $\alpha_3> \alpha_\nu$, where we recall that if $\alpha_\nu =2$ then $\alpha_2 =2$.
We now need to pick the optimal $q$ which maximises the above rate of convergence. By linear programming, the optimal $q$ is the minimum of the solution to $q\zeta_2 = \zeta_3 - q$ and $q\zeta_2 = \zeta_4 -q\zeta_5$, that is, the optimal $q$ is given by
\begin{align*}
q^* = \min\left(\frac{\zeta_3}{\zeta_2 + 1} , \frac{\zeta_4}{\zeta_2 + \zeta_5}\right).
\end{align*}
From the above computations, we deduce 
$$
\max_{1\leq i \leq N}\sup_{0\leq t \leq T}\e[|X^{i,N}_t - X^{n, i, N}_t|] \leq \frac{C}{n^{\zeta_2 q^{*}}}. 
$$
Let us now discuss the optimal value of the parameters $(\alpha_1, \alpha_2, \alpha_3)$ which maximizes the above rate of convergence. We distinguish two different cases: $\alpha_\nu < 2 (1-\gamma)/(1-\eta)$ and $\alpha_\nu \geq 2 (1-\gamma)/(1-\eta)$. 
In the first case $\alpha_\nu < 2 (1-\gamma)/(1-\eta)$, we select $\alpha_1 = 2 (1-\gamma)/(1-\eta)$ so that $\zeta_2= 2\gamma-1$. We now consider the two following sub-cases: $\eta\leq \gamma$ (or equivalently $2(1-\gamma)/(1-\eta)\leq 2$) and $\eta > \gamma$ (or equivalently $2(1-\gamma)/(1-\eta)> 2$).  In the first sub-case $\eta\leq \gamma$ one has $2\eta \leq 2 \gamma < \beta_\nu$ which implies $2 \leq 2\gamma/\eta < \beta_\nu /\eta$. We thus take $\alpha_2 = 2$ so that $\zeta_3=\eta$ which in turn implies $\zeta_3/(\zeta_2+1) = \eta/(2\gamma)$. We finally pick $\alpha_3 = \alpha_1$ so that $\zeta_2 + \zeta_5 \leq \eta$ which in turn implies that $\zeta_4/(\zeta_2+\zeta_5)\geq 1/2$. Hence, if $\eta\leq \gamma$, one has $\zeta_2 q^*=\eta - \eta/(2\gamma)$.
Now, in the second sub-case $\eta > \gamma$, we can pick $\alpha_2 \in  [2\gamma/\eta , 2 \wedge (\beta_\nu/\eta))$ so that $\zeta_3=\gamma$ which in turn implies $\zeta_3/(\zeta_2+1) =1/2$. With the same choice of $\alpha_3$, we thus obtain $\zeta_2 q^*=\gamma - 1/2$. To sum up, if $\alpha_\nu \leq 2 (1-\gamma)/(1-\eta)$, one obtains a convergence rate of order $n^{-(\gamma-\frac12)}$ if $\gamma \leq \eta$ and $n^{-(\eta-\frac{\eta}{2\gamma})}$ if $\eta < \gamma$.


We now turn our attention to the second case, namely $\alpha_\nu \geq 2 (1-\gamma)/(1-\eta)$. Note that since $\alpha_\nu \in [1,2]$ a necessary condition is $2 (1-\gamma)/(1-\eta) \leq 2$ which is equivalent to $\eta \leq \gamma$. Then, one has $\zeta_2+1 = 2-\alpha_1(1-\eta)$. Since $\beta_\nu > 2 \gamma \geq 2 \eta$, we select $\alpha_2=2$ so that $\zeta_3 = \gamma \wedge \eta = \eta$. The maps $\alpha_1 \mapsto \zeta_2(\alpha_1) \zeta_3/(\zeta_2(\alpha_1)+1)$ and $\alpha_1 \mapsto \zeta_2(\alpha_1) \zeta_4/(\zeta_2(\alpha_1)+\zeta_5)$ are decreasing so that we select $\alpha_1$ as small as possible, namely $\alpha_1= \alpha_\nu + \delta/(1-\eta)$, with $\delta \in (0, 1-(1-\eta)\alpha_\nu)$ and we denote this choice of $\alpha_1$ by $\alpha^{+}_\nu$ in order to save notation. This in turn yields $\zeta_2 \zeta_3/ (\zeta_2+1) = \eta(1-\alpha_\nu^{+}(1-\eta))/(2-\alpha_\nu^+(1-\eta))$. We finally select $\zeta_5=(\alpha_3-1)(1-\eta)=\frac{1}{2}(\zeta_2+1)-\zeta_2$, that is, $\alpha_3 = 1+\frac12 \alpha_\nu^{+}$. Note that this choice is admissible since $\alpha_3 > \alpha_\nu$ is equivalent to $2+\delta/(1-\eta)>\alpha_\nu$. Hence, one has $\zeta_2 \zeta_4/(\zeta_2+\zeta_5)= \zeta_2 \zeta_3/(\zeta_2+1)$. We thus obtain a convergence rate of order 
$$
n^{- \eta\frac{(1-\alpha_\nu^{+}(1-\eta))}{2-\alpha_\nu^+(1-\eta)}}= n^{-\eta(1- \frac{1}{2-\alpha_\nu(1-\eta)-\delta})}
$$
\noindent for any $\delta \in (0, 1-(1-\eta)\alpha_\nu)$.

We conclude by investigating the case $\gamma = 1/2$. Coming back to \eqref{EP_13}, under the constraints $\alpha_1 \in (\alpha_\nu, 1/(1-\eta))$, $\alpha_2 \in (\alpha_\nu, \beta_\nu/\eta) \cap (\alpha_\nu , 2]$ and $\alpha_3> \alpha_\nu$, we get
	\begin{align*}
	\max_{1\leq i\leq N}\e[|Y_t^{n, i}|]  & \leq C
	\left\{
	\left(1+ \frac{1}{\log(\delta)}\right)\varepsilon
	+ 	\frac{1}{n^{1/2}}
	+\frac{\varepsilon^{2\gamma-1}}{\log(\delta)}
	+ \frac{\varepsilon^{1-{\alpha_1}(1-\eta)}}{(\log(\delta))^{\alpha_1-1}}
	+\frac{1}{n^{\rho/2}} 
	+\frac{\delta}{\varepsilon \log(\delta) }\frac{1}{n^{\gamma}}\right.\\
	& \quad  \left.+ \frac{\delta}{\varepsilon \log(\delta)} \frac{1}{n^{\alpha_2\eta/2}} + \left(1+ \frac{1}{\log(\delta)}\right)\frac{1}{n^{\eta/2}} + u^{2-\alpha_3}\left(\frac{1}{\varepsilon^{1-\eta}\log(\delta)} + \frac{1}{u}\right)\frac{1}{n^{\eta/2}} 
	\right\},
	\end{align*}
\noindent for any $u\in (0, r)$, $r>0$ being fixed, we have again used the fact that by \eqref{left:right:tail:asymptotics} we have $\sup_{u \in (0, r)} \big[u^{\alpha_3-2} \int_0^{u} z^2 \nu(dz) + u^{1-\alpha_3} \int_{u}^{\infty} z \nu(dz)\big] < \infty$.
To proceed, we pick $u = u_n = \varepsilon^{1-\eta}\log(\delta)$, $\varepsilon=n^{-q}$ and $\delta =n^{p}$ with $p, q>0$. Note that taking $n$ large enough, one has $u_n \in (0,r)$. We then select the couple $(p, q) \in (\R_+)^2 \backslash{(0,0)}$ such that $p+q \leq \frac{\alpha_2\eta}{2} \wedge \gamma$ and $\frac{\eta}{2} - q(1-\eta)(\alpha_3-1) > 0$. For instance, if $\beta_\nu \leq 2 \eta$, as previously done choose $\alpha_2$ large enough so that $2\gamma < \alpha_2 \eta < \beta_\nu$, any $\alpha_1 \in (\alpha_\nu, (1-\eta)^{-1})$, $\alpha_3 = 1 + (1-\eta)^{-1}\gamma^{-1} \eta$ and then select $p= q < \gamma/2$. If $\beta_\nu > 2\eta$, choose $\alpha_2 = 2$, any $\alpha_1 \in (\alpha_\nu, (1-\eta)^{-1})$, $\alpha_3 = 1 + (1-\eta)^{-1}(\gamma \wedge \eta)^{-1} \eta$ and then select $p= q < (\gamma\wedge \eta)/2$. Hence, with this choice of parameters, we obtain
	\begin{align*}
	\max_{1\leq i \leq N}\e[|Y_t^{n,i}|]
	\leq \frac{C}{\log(n)}
	\end{align*}
\noindent for some positive constant independent of $n$ and $N$. The proof is now complete.
\end{proof}

\section{Positivity of the solution and applications}
Finally, in this section, we briefly mention some potential applications and conditions under which the solution of the McKean-Vlasov SDE \eqref{mckean:sde} is positive. A possible application we have in mind is the mean-field extension of the following equation
\begin{gather*}
X_t = \xi + \int_0^t b(X_s) ds + \int_0^t \sigma_1\sqrt{X_s}dW_s + \int_0^t \sigma_2\sqrt[\alpha]{X_{s-}}dZ_s,
\end{gather*}
where $W$ is a Brownian motion and $Z$ is a spectrally positive $\alpha$-stable process with index $\alpha \in (1,2]$. In particular, when $b(x) = (a-kx)$ where $a$ and $k$ are some positive constants, the above equation have recently attracted large attention in financial modelling. We refer to the works of Jiao et al. \cite{JMS, JMSS, JMSZ} for the modelling of sovereign interest rate, electricity prices and stochastic volatility. For multi-curve term structure models we refer to Fontana et al. \cite{FGS}. 

The additional introduction of a dependence with respect to the law in the drift $b$ is motivated by the fact that in the modelling of large credit portfolio or more generally of large financial markets, one tractable way to understand systemic risk can be done through the lens of interacting particle systems. Indeed, it has now become more and more important to capture not only self-exciting effects but also contagious effects, i.e. how the default of one firm affects another firm with the aim of shedding light on macro-aspects of how crises emerge in large interconnected systems. Researches in this direction have been considered e.g. by Giesecke et al. \cite{GSS, GSSS}, where the authors have introduced a model for the default intensities of $N$ firms, through the following system of interacting particles
\begin{align*}
d\lambda^i_t &= (a-k\lambda^i_t)dt + \sigma_1\sqrt{\lambda^i_t}dW_t^i + \sum_{j\neq i}^N c^{N}_j dH^j_t + \sigma_2\lambda^i_t dX_t
\end{align*}

\noindent where $X_t$ is some common risk factor which affects all firms, $W=(W^i)_{i=1,\dots, N}$ is an $N$-dimensional Brownian motion independent of $X$, $(c^{N}_j)_{1\leq j \leq N}$ are positive constants,
\begin{align*}
\tau^i & :=  \inf \{ t :  \int^t_0 \lambda_s^i ds > U^i \},\quad \mathrm{and} \quad H^i_t  := \mathbf{1}_{\{ \tau^i \leq t\}}
\end{align*}
\noindent with the random variables $(U^i)_{1\leq i \leq N}$, being i.i.d. with exponential distribution.  

The contagion effect is here modeled through the introduction of the sum $\sum_{j\neq i}^N c_j dH^j_t$ which induces a change in the default intensity of the $i$-th firm when other firms default. We believe that the mean-field dynamics presented and analyzed in this work could be used to develop extensions of the model proposed in the aforementioned references by introducing positive jumps, interaction in the default intensities between the firms not only at the time of defaults, but also instantaneously in time through the drift term. Obviously this is not within the scope of the current paper. We postpone to future works for a more specific discussion of the stylized features of such models.

To conclude, we provide here some sufficient conditions on the coefficients under which the McKean-Vlasov SDE \eqref{mckean:sde} is non-negative.
\begin{Ass}\label{Ass_3}
For any fixed $(t, \mu) \in [0,\infty) \times \mathcal{P}_1(\mathbb{R})$ such that $\mu(\R_+)=1$, the coefficients $b(t,\cdot,\mu)$, $\sigma(\cdot)$ and $h(\cdot)$ satisfy, $\sigma(t, 0) = h(t, 0) =0$ and $b(t, 0, \mu)\geq 0$.
\end{Ass}

\begin{Prop}\label{positive:solutions:mckean:sde}
 Under the same assumptions as in Theorem \ref{strong:existence:uniqueness} and assumption \ref{Ass_3}, for any initial condition $\xi$ such that $\mathbb{P}(\xi \geq 0) =1$, the unique strong solution $X=(X_t)_{t\geq0}$ of the McKean-Vlasov SDE \eqref{mckean:sde} satisfies $\mathbb{P}(X_t \geq 0, \, \forall t\geq 0)=1$. 
\end{Prop}

\begin{proof}
For the initial condition $P^{(0)}(t) = \mu$, $t\geq0$, we consider the sequence $\big\{X^{(m)}= (X^{(m)}_t)_{t\geq0}, m\geq1\big\}$ given by the dynamics \eqref{approximation:mckean:sde}. Since each $X^{(m)}$ is given by the solution of a standard SDE with positive jumps and time inhomogeneous drift coefficient $\widetilde{b}_m(t, x) := b(t, x, [X^{(m-1)}_t])$, by induction on $m\geq1$, it follows from Proposition 2.1 of \cite{Fu:Li} that $\mathbb{P}(X^{(m)}_t \geq 0, \, \forall t \geq0) = 1$. Up to the extraction of a convergent subsequence, we may assume that $(X^{(m)})_{m\geq 1}$ weakly converges to the unique strong solution $X$ of the SDE \eqref{mckean:sde}. We thus deduce that $\mathbb{P}(X_t\geq0) = \lim_{m} \mathbb{P}(X^{(m)}_t \geq0) = 1$. Now, applying again Proposition 2.1 of \cite{Fu:Li} to the linearized SDE \eqref{mckean:sde}, with drift $\widetilde{b}(t, x) := b(t, x, [X_t])$, allows to conclude the proof.  
\end{proof}

\begin{Eg}
As a toy model, the default intensities can interact through the average intensity, that is
\begin{align*}
d\lambda^i_t & = \kappa_t (\overline{\lambda}_t - k_t \lambda^i_t)dt + \sigma_1\sqrt[r]{|\lambda^i_t|}dW^i_t + \sigma_2 \sign(\lambda^{i}_{s-})\sqrt[q]{|\lambda^i_{s-}|}dZ^i_s
\end{align*}
where $(W^{i},Z^{i})_{1\leq i \leq N}$ are i.i.d. copies of $(W, Z)$, $W$ and $Z$ being respectively a one-dimensional Brownian motion and a spectrally-positive $\alpha$-stable process with index $\alpha \in (1, 2)$ such that $1/q +1/\alpha \geq 1$ and  $q\geq1$, $1\leq r \leq 2$. Here, $\overline{\lambda}_t = N^{-1}\sum_{j=1}^N \lambda^j_t$ is the average intensity, $\sigma_1, \, \sigma_2$ are some non-negative constants and $t\mapsto \kappa_t, \, k_t$ are deterministic functions, $t\mapsto \kappa_t$ being non-negative, which measure the strength of the mean-reversion. As the number of entities $N$ goes to infinity, the $N$-interacting particle system will converge to the mean-field equation with dynamics 
\begin{align*}
d\lambda_t & = \kappa_t (\mathbb{E}[\lambda_t]-k_t\lambda_t)dt + \sigma_1\sqrt[r]{|\lambda_t|}dW_t + \sigma_2\sign(\lambda_{s-})\sqrt[q]{|\lambda_{s-}|}dZ_s.
\end{align*}

Combining Theorem \ref{strong:existence:uniqueness} and Proposition \ref{positive:solutions:mckean:sde}, we deduce that for any non-negative initial condition $\xi$ with law $\mu$ such that $M_1(\mu) = \int_{\R_+} x^\beta \mu(dx) < \infty$, for some $\beta>1$, the above mean-field SDE admits a unique non-negative strong solution. 	Note that taking expectation on both hand side of the above equation, the mean can be written explicitly namely
\begin{gather*}
\mathbb{E}[\lambda_t] = \mathbb{E}[\xi] + \int^t_0 \mathbb{E}[\lambda_s]\kappa_s(1-k_s) ds =  \mathbb{E}[\xi] \exp\Big(\int^t_0 \kappa_s (1-k_s) ds\Big).
\end{gather*}
Note that instead of $h(x) = \sigma_2\sign(x)\sqrt[\alpha]{|x|}$, one can take $h(x) = \sigma_2\sqrt[\alpha]{|x|^+}$ in the particle system. In the limit, the resulting mean-field equation will have the same dynamics.
\end{Eg}

\section{Appendix}\label{appendix:section}

\subsection{Weak existence and moment estimates for some jump-type non-linear SDE}\label{weak:existence:lemma}
We establish here the weak existence as well as a moment estimate under mild assumptions on the coefficients for some general jump-type non-linear SDE with dynamics
\begin{equation}
X_t  = \xi + \int_0^t b(s, X_{s}, [X_s]) \, ds + \int_0^t \sigma(s, X_{s}) \, dW_s + \int_0^t  h(s, X_{s-}) \, dZ_{s}
\label{general:sde:mckean}
\end{equation}

\noindent where $W$ is a one-dimensional Brownian motion, $Z$ is a compensated Poisson random measure independent of $W$ with intensity measure $ds\nu(dz)$ satisfying $\int_{\R\backslash{\left\{0\right\}}} (|z|\wedge z^2) \nu(dz) < \infty$ and with starting point $\xi$ independent of $W$ and $Z$. The following result seems to be quite standard but we were not able to find a proof.

\begin{Lem}\label{weak:solution:and:moment:estimate} Assume that $b$, $\sigma$ and $h$ have at most linear growth in $x$ and $\mu$, locally uniformly in $t$, in the sense of \eqref{linear:growth:condition}.
Assume that $\int_{|z|\geq1} |z|^{\beta} \, \nu(dz) < \infty$ for some $\beta \geq1$ and that the initial distribution $\mu \in \mathcal{P}(\R)$ of the mean-field SDE \eqref{general:sde:mckean} has a finite $\beta$-moment, that is, $M_\beta(\mu)=\int_{\R}|x|^\beta \mu(dx) < \infty$.
Then, the following statements hold:
\begin{enumerate}
\item[(i)] For any weak solution to \eqref{general:sde:mckean} with starting distribution $\mu$, for any $T>0$, one has
\begin{equation}
\mathbb{E}[\sup_{0\leq t\leq T} |X_t|^\beta] < \infty. \label{moment:estimate:jump:type:mckean:sde}
\end{equation}
\item[(ii)] Assume that the aforementioned assumptions hold with $\beta>1$. If $[0,\infty) \times \R \times \mathcal{P}_1(\R) \ni (t, x, \mu) \mapsto b(t, x, \mu), \, \sigma(t, x), \, h(t, x)$ are continuous functions, $\mathcal{P}_1(\R)$ being equipped with the Wasserstein metric $W_1$, then there exists a weak solution to \eqref{general:sde:mckean}.  
\end{enumerate}
\end{Lem}

\begin{proof} 
\textit{Step 1:} We first prove the moment estimate \eqref{moment:estimate:jump:type:mckean:sde}. We restrict ourself to the case $\beta \in [1,2]$. The case $\beta >2$ can be treated in a completely analogous manner. We adapt to our current mean-field setting the argument of Proposition 2 of Fournier \cite{Fournier}. We introduce the auxiliary equation
\begin{equation}
\label{reduced:form:mckean:vlasov}
Y_t = \xi +  \int_0^t \sigma(s, Y_s) \, dW_s + \int_0^t \int_{|z|\leq 1} h(s, Y_{s-}) z \widetilde{N}(ds, dz) + \int_0^t c(s, Y_{s}, [Y_s]) \, ds 
\end{equation}

\noindent with $c(s, x, \mu) = b(s, x, \mu) - h(s, x) \int_{|z|\geq1} z\nu(dz)$ and $\xi$ is a real-valued random variable independent of $(W, N)$ with distribution $\mu$ satisfying $M_\beta(\mu)= \int_{\R} |x|^\beta \mu(dx)< \infty$. Since $b, \, \sigma, \, h$ have at most linear growth at infinity in $x$ and $\mu$ locally uniformly w.r.t. the time variable $t$ and $\int_{|z|\leq 1} z^2 \, \nu(dz) < \infty$, it is easily checked that for any $T>0$, there exists a constant $C_T$ such that
$$
\mathbb{E}\big[\sup_{0\leq t \leq T} Y_t^2 \big| \xi  \big] \leq C_T\left(1+ \xi^2 + \int_0^T\mathbb{E}[Y^2_s|\xi] \, ds  + \left(\int_0^T W_1([Y_s], \delta_0) \, ds\right)^2\right) 
$$

\noindent which in turn by Gr\"onwall's lemma yields
\begin{align*}
\mathbb{E}\big[\sup_{0\leq t \leq T} Y^2_t\big| \xi\big]^{\frac12} &  \leq C_T\bigg(1+ |\xi| + \int_0^T \mathbb{E}[|Y_s|] \, ds\bigg).
\end{align*}

Hence, by Jensen's inequality
\begin{equation}
\label{moment:estimate:reduced:form:mckean:vlasov}
\mathbb{E}\big[\sup_{0\leq t \leq T} Y^2_t\big| \xi\big]^{\frac\beta2} \leq C_T\bigg(1+ |\xi|^\beta + \int_0^T\mathbb{E}[|Y_s|^\beta] \, ds\bigg)
\end{equation}


\noindent so that, by taking expectation and applying again Jensen's inequality
\begin{equation}
\label{beta:moment:estimate:reduced:form:mckean:vlasov}
\mathbb{E}[\sup_{0\leq t \leq T} |Y_t|^\beta] \leq C_T\bigg(1+ M_\beta(\mu) +  \int_0^T\mathbb{E}[|Y_s|^\beta] \, ds\bigg)
\end{equation}

\noindent where we recall that $M_\beta(\mu) = \int_{\R} |x|^\beta d\mu(x)$.
Observe now that the dynamics \eqref{general:sde:mckean} can be rewritten as
\begin{align}
X_t & = \xi +  \int_0^t \sigma(s, X_s) \, dW_s +  \int_0^t \int_{|z|\leq 1} h(s, X_{s-}) \, z \widetilde{N}(ds, dz) \label{general:nonlinear:jump:sde:rewrite}\\
& \quad + \int_0^t c(s, X_s, [X_{s}]) \, ds +  \int_0^t \int_{|z|\geq1} h(s, X_{s-})\, z N(ds, dz). \nonumber
\end{align}

The last integral of \eqref{general:nonlinear:jump:sde:rewrite} generates jumps at discrete instants. More precisely, one may write the restriction of $N$ to $[0, \infty) \times \left\{ |z|\geq1\right\}$ as $\sum_{n\geq1} \delta_{(T_n, Z_n)}$ where the $(T_n)_{n\geq1}$ are the jump times of a Poisson process denoted in what follows by $J$ with parameter $\lambda = \int_{|z|\geq1}\nu(dz)$ and where the random variables $(Z_n)_{n\geq1}$ are i.i.d. with law $\lambda^{-1}\textbf{1}_{\left\{|z|\geq 1\right\}} \nu(dz)$. We thus see that conditioning w.r.t the $\sigma$-field $\mathcal{G}=\sigma(T_n, n\geq1)$ the dynamics \eqref{general:nonlinear:jump:sde:rewrite} reduces to \eqref{reduced:form:mckean:vlasov} on each time interval $(T_n, T_{n+1})$. Namely, $X$ solves \eqref{reduced:form:mckean:vlasov} on $[0,T_1)$. Hence, by \eqref{beta:moment:estimate:reduced:form:mckean:vlasov}
$$
\mathbb{E}\big[\sup_{0\leq t < T_1 \wedge T}|X_t|^{\beta}\big| \mathcal{G}\big] \leq C_T\bigg(1+M_\beta(\mu) + \int_0^{T_1\wedge T} \mathbb{E}[|X_s|^\beta]\, ds \bigg).
$$

The preceding upper-bound together with \eqref{beta:moment:estimate:reduced:form:mckean:vlasov} yield
\begin{align*}
\mathbb{E}\big[\sup_{0\leq t \leq T_1 \wedge T}|X_t|^{\beta}\big| \mathcal{G}\big] &  \leq \mathbb{E}\big[\sup_{0\leq t \leq T}|X_t|^{\beta}\big| \mathcal{G}\big] \textbf{1}_{\left\{T_1\geq T\right\}} + \mathbb{E}\big[\sup_{0\leq t < T_1}|X_t|^{\beta}\big| \mathcal{G}\big]  \textbf{1}_{\left\{T_1< T\right\}} + \mathbb{E}\big[|X_{T_1}|^{\beta}\big| \mathcal{G}\big]  \textbf{1}_{\left\{T_1< T\right\}} \\
& \leq C_T\bigg(1+M_\beta(\mu)+ \int_0^{T_1\wedge T} \mathbb{E}[|X_s|^\beta]\, ds \bigg) + \mathbb{E}\big[|X_{T_1}|^{\beta}\big| \mathcal{G}\big]  \textbf{1}_{\left\{T_1< T\right\}} .
\end{align*}
\noindent Now, $X_{T_1} = X_{T_1-}+h(T_1, X_{T_1-}) Z_1$ so that, using the linear growth assumption of the jump coefficient (uniformly on $[0,T]$), $|X_{T_1}|^\beta \leq C_T(1+|X_{T_1-}|^\beta)(1+|Z_1|^\beta)$ on the set $\left\{ T_1 < T\right\}$. Again by \eqref{beta:moment:estimate:reduced:form:mckean:vlasov}$,\mathbb{E}[|X_{T_1-}|^\beta | \mathcal{G}] \leq C_T(1+M_\beta(\mu) + \int_0^{T_1 \wedge T} \mathbb{E}[|X_s|^\beta] \, ds)$ which in turn clearly implies
$$
\mathbb{E}[|X_{T_1}|^\beta | \mathcal{G}] \leq C_T(1+\mathbb{E}[|Z_1|^\beta])  \bigg(1+M_\beta(\mu) + \int_0^{T_1 \wedge T} \mathbb{E}[|X_s|^\beta] \,ds\bigg)
$$

\noindent on the set $\left\{ T_1 < T\right\}$. Hence, up to a change of the positive constant $C_T$ (which from now on depends on $\mathbb{E}[|Z_1|^\beta]:=\int_{|z|\geq1}|z|^\beta \nu(dz) < \infty$), we therefore deduce the following estimate
$$
\mathbb{E}\big[\sup_{0 \leq t \leq T_1 \wedge T} |X_t|^\beta\big| \mathcal{G}\big]\leq C_T\bigg(1+ M_\beta(\mu) +  \int_0^{T_1 \wedge T} \mathbb{E}[|X_s|^\beta] \,ds \bigg).
$$

Similarly, on each time interval $(T_k, T_{k+1})$, $k\geq1$, one has
\begin{align*}
\mathbb{E}\big[\sup_{ T_{k}\wedge T\leq t \leq T_{k+1} \wedge T} |X_t|^\beta\big| \mathcal{G}\big] & \leq C_T\bigg(1+ \mathbb{E}[|X_{T_k\wedge T}|^\beta|\mathcal{G}] + \int_{T_{k}\wedge T}^{T_{k+1} \wedge T} \mathbb{E}[|X_s|^\beta] \,ds \bigg) \\
& \leq C_T\bigg(1+ \mathbb{E}\big[\sup_{T_{k-1}\wedge T\leq t \leq T_{k} \wedge T}|X_t|^\beta\big| \mathcal{G}\big] +  \int_{T_{k}\wedge T}^{T_{k+1} \wedge T} \mathbb{E}[|X_s|^\beta] \,ds \bigg)
\end{align*}

\noindent where $C_T$ is a positive constant that does not depend on the index $k$. We thus deduce that there exists a constant $K_T>1$ (which depends on $M_\beta(\mu)$ and $T$) such that 
$$
\mathbb{E}\big[\sup_{T_{k}\wedge T\leq t \leq T_{k+1} \wedge T} |X_t|^\beta\big| \mathcal{G}\big] \leq K^{k+1}_T \left(1 + \int_0^{T_{k+1}\wedge T}  \mathbb{E}[|X_s|^\beta] \,ds\right) 
$$
\noindent which in turn implies
\begin{align*}
\mathbb{E}\big[\sup_{0 \leq t \leq T_{k} \wedge T} |X_t|^\beta\big| \mathcal{G}\big] &  \leq \mathbb{E}\big[\max_{1\leq j\leq k}\sup_{T_{j-1}\wedge T \leq t \leq T_{j} \wedge T} |X_t|^\beta\big| \mathcal{G}\big] \\
& \leq (K_T + \cdots + K^{k}_T)  \left(1 + \int_0^{T_{k}\wedge T}  \mathbb{E}[|X_s|^\beta] \,ds\right)  \\
& \leq \frac{K^{k+1}_T}{K_T-1} \left(1 + \int_0^{T_{k}\wedge T}  \mathbb{E}[|X_s|^\beta] \,ds\right) .
\end{align*}

Hence, using the fact that $\left\{J_T = n \right\} = \left\{T_n \leq T < T_{n+1}\right\}$ together with the above estimate, we conclude
\begin{align*}
\mathbb{E}[\sup_{0\leq t \leq T}|X_t|^\beta] & = \sum_{n\geq0} \mathbb{E}[\sup_{0\leq t \leq T}|X_t|^\beta \textbf{1}_{\left\{J_T=n\right\}}] \\
& =  \sum_{n\geq0} \mathbb{E}\big[\textbf{1}_{\left\{T_n\leq T < T_{n+1}\right\}} \mathbb{E}\big[\sup_{0\leq t \leq T_{n+1}\wedge T}|X_t|^\beta |\mathcal{G}\big] \big]\\
& \leq \frac{1}{K_T-1} \sum_{n\geq0} K^{n+2}_T \mathbb{E}\bigg[\textbf{1}_{\left\{T_n\leq T < T_{n+1}\right\}} \bigg(1+ \int_0^{T_{n+1}\wedge T}  \mathbb{E}[|X_s|^\beta] \,ds\bigg) \bigg]\\
& \leq \frac{1}{K_T-1} \sum_{n\geq0} K^{n+2}_T \frac{(\lambda T)^{n}}{n!} e^{-\lambda T} \bigg(1+ \int_0^{T}  \mathbb{E}[|X_s|^\beta] \,ds\bigg).
\end{align*}
Finally, one concludes the proof of the moment estimate \eqref{moment:estimate:jump:type:mckean:sde} by using again Gr\"onwall's inequality.
\vskip1pt

\noindent \textit{Step 2:} We now prove that weak existence holds for the SDE \eqref{general:sde:mckean} under the additional assumptions that $M_\beta(\mu) + \int_{|z|\geq1} |z|^{\beta} \, \nu(dz) < \infty$, for some $\beta>1$ and that the coefficients $b$, $\sigma$ and $h$ are continuous. We proceed using a compactness argument on the space of probability on $\mathcal{D}([0,\infty), \R)$. We thus aim at constructing a sequence of probability measures on $\mathcal{D}([0,\infty), \R)$ that will converge (up to the extraction of a subsequence) to a solution of the corresponding martingale problem with associated infinitesimal generator $(\mathcal{L}^{P}_t)_{t\geq0}$.  A probability measure $P \in \mathcal{D}([0,\infty), \R)$, with time marginals $(P(t))_{t\geq0}$ is a solution to the martingale problem starting from $\mu$ at time $0$ if, denoting by $y=(y(t))_{t\geq0}$ the canonical process and $\mathcal{F}$ the canonical filtration, one has $P(y(0)\in \Gamma) = \mu(\Gamma)$, for all $\Gamma \in \mathcal{B}(\R)$ and for any $\varphi \in \mathcal{C}^{\infty}_b(\R)$
$$
f(y(t)) - \int_0^t \mathcal{L}^{P}_s\varphi(y(s))\, ds
$$

\noindent is an $(\mathcal{F},P)$-martingale where 
\begin{align*}
\mathcal{L}^{P}_t \varphi(x) & = b(t, x, P(t))\varphi'(x) + \frac{1}{2} \sigma^2(t, x, P(t)) \varphi''(x) \\
& \quad + \int_{\R} \left\{\varphi(x+ h(t, x, P(t))z) - \varphi(x) - \varphi'(x)h(t, x, P(t))z \textbf{1}_{\{|z|\leq 1\}} \right\} \nu(dz).
\end{align*}

 We consider the sequence of probability measures $(P^{(m)})_{m\geq0}$ on $\mathcal{D}([0,\infty), \R)$, constructed as follows: for a given probability measure $P^{(0)}$ on $\mathcal{D}([0,\infty), \R)$, such that $\lim_{s\rightarrow t}W_1(P^{(0)}(t),P^{(0)}(s)) = 0$ and $\sup_{s\geq 0} W_1(P^{(0)}(s), \delta_0) < \infty$, and for a given non-negative integer $m$, we let $P^{(m+1)}$ be the probability measure induced by a weak solution to the following SDE with dynamics
 \begin{equation}
 X^{(m+1)}_t  = \xi + \int_0^t b(s, X^{(m+1)}_{s}, P^{(m)}(s)) \, ds + \int_0^t \sigma(s, X^{(m+1)}_{s}) \, dW_s + \int_0^t  h(s, X^{(m+1)}_{s-}) \, dZ^{m}_{s} \label{approximation:mckean:sde}
 \end{equation} 

\noindent where $Z^{m}_t = \int_0^t \int_{\R\backslash{\left\{0\right\}}} z \widetilde{N}_{m}(ds,dz)$, $\widetilde{N}_m$ being a compensated Poisson random measure on $[0,\infty)\times \R\backslash{\left\{0\right\}}$ with intensity measure $dt \textbf{1}_{|z|\leq m} \nu(dz)$. Following similar lines of reasoning as those employed in the first step, one may prove that for any weak solution to \eqref{approximation:mckean:sde} with starting distribution $\mu$, for any $T>0$, one has
\begin{equation}
\sup_{m\geq1} \mathbb{E}[\sup_{0\leq t \leq T}|X^{(m)}_t|^\beta] < \infty.\label{supremum:sequence:moments}
\end{equation}
We deliberately omit the proof of the above estimate. Similarly, for any weak solution to \eqref{approximation:mckean:sde}, $m\geq0$, for any $0\leq s < t$, there exists a positive constant $C$ (possibly depending on $m$) such that
$$
\mathbb{E}[|X^{(m+1)}_t-X^{(m+1)}_s|] \leq C (t-s)^{1/2}.
$$
The previous bound directly stems from \eqref{approximation:mckean:sde} and some standard computations.
Now, by induction on $m\geq1$, using the fact that $\lim_{s\rightarrow t} W_1(P^{(m)}(t), P^{(m)}(s)) \leq \lim_{s\rightarrow t} \mathbb{E}[|X^{(m)}_t-X^{(m)}_s|] = 0$, the continuity of the coefficients $b$, $\sigma$ and $h$, and $\int_{|z|\leq m} |z|^2 \nu(dz) < \infty$, from Theorem 175 of Situ \cite{Situ}, weak existence holds for the SDE \eqref{approximation:mckean:sde} for any $m\geq0$.

We now establish the tightness on $\mathcal{D}([0,\infty), \R)$ of the sequence $(P^{(m)})_{m\geq0}$ by employing Aldou's criterion. Let $(S, S')$ be two stopping times satisfying $a.s.$ $0\leq S \leq S' \leq S+\delta \leq T$, for some $\delta>0$. We also introduce for some constant $A>0$ (to be chosen later on) the stopping time $\tau_{A} = \inf\left\{ t\geq 0: |y(t)| \geq A\right\}$. For any $\varepsilon$, take $\phi \in \mathcal{C}^{\infty}_b(\R)$ so that $0\leq \phi \leq 1$, $\phi(0)=0$ and $\phi \equiv 1$ for $|x|\geq\varepsilon$. Given an integer $m$ and a stopping time $\tau$, we let $P^{(m)}_{\tau, w}$ be the regular conditional probability distribution of $P^{(m)}$ given $\mathcal{F}_{\tau}$. Then, 
$$
\phi(y(t\wedge S') - y((S\wedge \tau_A)(w))) - \int_{(S\wedge \tau_A)(w)}^{t\wedge S'} \mathcal{L}^{P^{(m)}}_r\phi(y(r)-y(S(w))) \, dr
$$
\noindent is a $P^{(m)}_{S\wedge \tau_A, w}$-martingale so that
\begin{align*}
P^{(m)}_{S,w}(|y(S')-y(S(w))| > \varepsilon, \tau_A \geq T)&  = P^{(m)}_{S,w}(|y(S'\wedge \tau_A)-y((S\wedge \tau_A)(w))| > \varepsilon, \tau_A \geq T) \\
& \leq \mathbb{E}^{P^{(m)}_{S,w}}[\phi(y(S'\wedge \tau_A) - y(S(w)\wedge \tau_A))]\\
& \leq \mathbb{E}^{P^{(m)}_{S,w}}\bigg[\int_{S(w)}^{(S(w) +\delta) \wedge \tau_A} |\mathcal{L}^{P^{(m)}}_r\phi(y(r) - y(S(w)))| \, dr\bigg] \\
& \leq C(1+A^2) \delta
\end{align*}

\noindent where we used the fact that $P^{(m)}$-a.s., for all $s \in [0, \tau_A\wedge T]$, 
$$
 |b(s, y(s-), P^{(m-1)}(s)| + |\sigma(s, y(s-))|^2 + |h(s, y(s-))|^2 \leq K(1+ A^2 + \sup_{m\geq 1, t\in [0,T]} W^2_1(P^{(m-1)}(t), \delta_0)),
$$ 
\noindent and where the constant $K$ depends only on the $C^2$-norm of $\phi$. Note that $\sup_{m\geq0, t\in [0,T]} W_1(P^{(m)}(t), \delta_0)$ is finite by \eqref{supremum:sequence:moments}. Also, again from \eqref{supremum:sequence:moments}, 
$$
P^{(m)}(\tau_A \leq T) \leq P^{(m)}(\sup_{0\leq t \leq T}|y(s)| > A) \leq \frac{C_T}{A}.
$$

From the above computations, we thus derive
$$
P^{(m)}(|y(S')-y(S)| \geq \varepsilon) \leq C(1+A^2) \delta + \frac{C_T}{A}
$$
\noindent so that by choosing $A=\delta^{-1/3}$, for all $\delta \in (0,1)$, we get
$$
P^{(m)}(|y(S')-y(S)| \geq \varepsilon) \leq C_T \delta^{\frac13}.
$$

\noindent where the constant $C_T$ does not depend on $m$ and $\delta$. We thus conclude that for all $T>0$, for all $\varepsilon >0$
$$
\lim_{\delta \downarrow 0} \limsup_{m} \sup_{S,S': S \leq S' \leq S+\delta \leq T}P^{(m)}(|y(S')-y(S)| \geq \varepsilon) = 0.
$$
As a consequence, by Theorem 4.5 on page 356 of Jacod and Shiryaev \cite{Jacod:Shiryaev}, the sequence of probability measures $(P^{(m)})_{m\geq0}$ is tight. By the Prokhorov theorem the sequence $(Z^{m})_{m\geq 0}$ is also tight since it converges weakly in $\mathcal{D}([0,\infty), \R)$ to $Z$. Relabelling the indices if necessary, we may assert that $(P^{(m)})_{m \geq 0}$ converges weakly to a probability measure $P^{\infty}$ and that $(X^{m}, Z^{m})_{m\geq 1}$ converges in law to $(X, Z)$ in $\mathcal{D}([0,\infty), \R^2)$.

From \eqref{supremum:sequence:moments} and the weak convergence of the sequence $(P^{(m)}(t))_{m\geq1}$ towards $P^{\infty}(t)$, by uniform integrability it follows that $\mathbb{E}[|X^{(m)}_t|] \rightarrow \mathbb{E}[|X_t|]$, for any $t\geq0$, so that the convergence of $(P^{(m)}(t))_{m\geq1}$ also holds with respect to the $W_1$ metric. Now, by continuity of $b$, $\sigma$ and $h$, the continuous mapping theorem implies that the family 
$$
(X^{(m+1)}_t, b(t, X^{(m+1)}_t, P^{(m)}(t)), \sigma(t, X^{(m+1)}_t), h(t, X^{(m+1)}_t), W_t, Z^{m}_t)_{t\geq0}, \, m\geq0
$$ 

\noindent converges in law to $(X_t, b(t, X_t, P^{\infty}(t)), \sigma(t, X_t), h(t, X_t), W_t, Z_t)_{t\geq0}$ in $\mathcal{D}([0,\infty), \R^{6})$.

We now employ Corollary 6.30, page 385 of Jacod and Shiryaev \cite{Jacod:Shiryaev} to deduce that the sequence $(Z^m)_{m\geq0}$ is Predictably Uniformly Tight (P-UT property) in the sense of Definition 6.1. page 377 of \cite{Jacod:Shiryaev}. In order to do that it suffices to check that $\sup_{m\geq0} \mathbb{E}[\sup_{s \in [0,t]}|\Delta Z^m_s|] < \infty$, for all $t>0$. One may bound the previous quantity by one (coming from the jumps which are smaller than one) plus the sum of jumps bigger than one. This yields
$$
\mathbb{E}[\sup_{s \in [0,t]}|\Delta Z^m_s|] \leq 1 + \mathbb{E}\bigg[\int_0^t \int_{1\leq |z|\leq m} |z| N_m(dsdz)\bigg] \leq 1 +t \int_{1\leq |z|\leq m} |z| \nu(dz) ds \leq  1 + Ct
$$

\noindent where $C=\int_{|z|\geq1}|z| \nu(dz) < \infty$. From Theorem 6.22, page 383 of Jacod and Shiryaev \cite{Jacod:Shiryaev}, the sequence 
$$
\bigg(X^{(m+1)}_t, \int_0^t b(s, X^{(m+1)}_s, P^{(m)}(s))\, ds, \int_0^t\sigma(s, X^{(m+1)}_s) \, dW_s, \int_0^t h(s, X^{(m+1)}_{s-}) \, dZ_s, W_t, Z^{m}_t\bigg)_{t\geq0},\, m\geq0
$$ 
\noindent converges in law to 
$$
\bigg(X_t, \int_0^t b(s, X_s, P^{\infty}(s))\, ds, \int_0^t\sigma(s, X_s) \, dW_s, \int_0^t h(s, X_{s-}) \, dZ_s, W_t, Z_t\bigg)_{t\geq0}.
$$

Hence, passing to the limit in the dynamics \eqref{approximation:mckean:sde}, we find that 
$$
X_t =  \xi + \int_0^t b(s, X_{s}, P^{\infty}(s)) \, ds + \int_0^t \sigma(s, X_{s}) \, dW_s + \int_0^t  h(s, X_{s-}) \, dZ_{s}
$$
in distribution,
\noindent so that $P^{\infty}$ is a weak solution to \eqref{general:sde:mckean}.
\end{proof}

\subsection{Yamada and Watanabe Approximation Technique}\label{YW}
To deal with the H\"older continuity of the coefficients $\sigma$ and $h$, we introduce below the Yamada and Watanabe approximation technique (see for example \cite{Gyongy:rasonyi,Li:Mytnik,yamada:watanabe} for some applications of this technique).
For each $\delta \in (1,\infty)$ and $\varepsilon \in (0,1)$, we select a continuous function $\psi _{\delta, \varepsilon}: \real \to \real^+$ with a support included in $[\varepsilon/\delta, \varepsilon]$ and such that
\begin{align*} 
\int_{\varepsilon/\delta}^{\varepsilon} \psi _{\delta, \varepsilon}(z) dz
= 1 \quad \text{ and } \quad  0 \leq \psi _{\delta, \varepsilon}(z) \leq \frac{2}{z \log \delta}, \:\:\:z > 0.
\end{align*}
Let us define the real-valued function $\phi_{\delta, \varepsilon} \in C^2(\real)$
\begin{align*}
\phi_{\delta, \varepsilon}(x)&:=\int_0^{|x|}\int_0^y \psi _{\delta, \varepsilon}(z)dzdy.
\end{align*}
It is straightforward to check that $\phi_{\delta, \varepsilon}$ satisfies the following useful properties: 
\begin{align} 
&|x| \leq \varepsilon + \phi_{\delta, \varepsilon}(x), \text{ for any $x \in \real $}, \label{phi3}\\ 
&0 \leq |\phi'_{\delta, \varepsilon}(x)| \leq 1, \text{ for any $x \in \real$} \label{phi2}, \\
&\phi'_{\delta, \varepsilon}(x) \geq 0, \text{ for } x\geq 0 \text{ and } \phi'_{\delta, \varepsilon}(x) < 0, \text{ for } x< 0, \label{phi1}\\
&\phi''_{\delta, \varepsilon}(\pm|x|)=\psi_{\delta, \varepsilon}(|x|)
\leq \frac{2}{|x|\log \delta}{\bf 1}_{[\varepsilon/\delta, \varepsilon]}(|x|)
\leq \frac{2\delta }{\varepsilon \log \delta},
\text{ for any $x \in \real \setminus\{0\}$}. \label{phi4}
\end{align}

We present below two technical lemmas. Lemma \ref{key_lem_0} below is analoguous to Lemma 3.2 given in \cite{Li:Mytnik}. We provide its proof for sake of completeness.
\begin{Lem}\label{key_lem_0}
	Suppose that the L\'evy measure $\nu$ satisfies $\int_0^{\infty} \{z \wedge z^2\} \nu(dz) < \infty$.
	Let $\varepsilon \in (0,1)$ and $\delta \in (1,\infty)$.
	Then for any $x \in \real$, $y \in \real \setminus\{0\}$ with $xy \geq 0$ and $u>0$, it holds that
	\begin{align*}
	&\int_{0}^{\infty}
	\{\phi_{\delta,\varepsilon}(y+xz)-\phi_{\delta,\varepsilon}(y)-xz\phi_{\delta,\varepsilon}'(y)\} \nu(dz)\notag
	\leq 2 \cdot \1_{(0,\varepsilon]}(|y|)
	\left\{
	\frac{|x|^2}{\log \delta} \left(\frac{1}{|y|} \wedge \frac{\delta}{\varepsilon}\right) \int_{0}^{u} z^2 \nu(dz)
	+ |x| \int_{u}^{\infty} z \nu(dz)
	\right\}.
	\end{align*}
\end{Lem}
\begin{proof}
	Let $x \in \real$, $y \in \real \setminus\{0\}$ with $xy \geq 0$ and $z >0$.
	By the second order Taylor's expansion for $\phi_{\delta,\varepsilon}$, it follows from \eqref{phi4} that
	\begin{align*}
	\phi_{\delta,\varepsilon}(y+xz)-\phi_{\delta,\varepsilon}(y)-xz\phi_{\delta,\varepsilon}'(y)
	&=|xz|^2 \int_{0}^{1} \theta \phi_{\delta,\varepsilon}''(y+ \theta xz)d \theta
	\leq \frac{2|xz|^2}{\log \delta} \int_{0}^{1} \frac{\theta \1_{[\varepsilon/\delta,\varepsilon]}(|y+ \theta xz|) }{|y+ \theta xz|} d \theta.
	\end{align*}
	Since $xy \geq 0$, we have $|y| \leq |y+\theta xz|$ and $\1_{[\varepsilon/\delta,\varepsilon]}(|y+ \theta xz|) \leq \1_{(0,\varepsilon]}(|y|)$.
	Hence we obtain
	\begin{align}\label{key_lem_2}
	\phi_{\delta,\varepsilon}(y+xz)-\phi_{\delta,\varepsilon}(y)-xz\phi_{\delta,\varepsilon}'(y)
	& \leq \frac{2 |xz|^2 \1_{(0,\varepsilon]}(|y|)}{ \log \delta} \left(\frac{1}{|y|} \wedge \frac{\delta}{\varepsilon}\right).
	\end{align}
	Moreover, since $xy \geq 0$, by \eqref{phi1} we have $x \phi_{\delta,\varepsilon}'(y) \geq 0$. This together with the fact that the right hand side of \eqref{key_lem_2} has $\1_{(0,\varepsilon]}(|y|)$, we obtain
	\begin{align}\label{key_lem_3}
	\phi_{\delta,\varepsilon}(y+xz)-\phi_{\delta,\varepsilon}(y)-xz\phi_{\delta,\varepsilon}'(y)
	&\leq \1_{(0,\varepsilon]}(|y|) \{\phi_{\delta,\varepsilon}(y+xz)-\phi_{\delta,\varepsilon}(y) \} \notag\\
	&=\1_{(0,\varepsilon]}(|y|) xz \int_{0}^{1}\phi_{\delta,\varepsilon}'(y+\theta xz) d \theta
	\leq \1_{(0,\varepsilon]}(|y|) |xz|.
	\end{align}
	The result then follows from \eqref{key_lem_2} and \eqref{key_lem_3}.
\end{proof}

\begin{Lem}\label{key_lem12}
Let $\varepsilon \in (0,1)$ and $\delta \in (1,\infty)$.
	Then, for any $\alpha \in (\alpha_\nu, 2]$, there exists some constant $C>0$ such that for any $u \in (0,\infty]$, for any $(x, x', y) \in \real^3$ satisfying $-\sign(y)x' \leq \kappa |y|$, for some positive constant $\kappa$ if $yx'<0$ it holds
	\begin{align*}
	&\int_{0}^{\infty}\left|
	\phi_{\delta,\varepsilon}(y+xz)-\phi_{\delta,\varepsilon}(y+x'z)-(x-x')z\phi_{\delta,\varepsilon}'(y)
	\right| \nu(dz) \leq C\bigg[\left\{\frac{\delta}{\varepsilon \log \delta} + 1\right\}|x-x'|^{\alpha}  +  |x-x'| \\
	& + \1_{\{y x'< 0\}}\left\{\frac{\kappa}{\log(\delta)} + 1\right\}|x-x'| + \1_{\{yx'\geq 0\}}\left\{\frac{\1_{(0,\varepsilon]}(|y|)}{\log \delta} \bigg(\frac{1}{|y|} \wedge \frac{\delta}{\varepsilon}\bigg) |x'| \int_0^u z^2 \nu(dz) + \int_u^{\infty} z \nu(dz)\right\}|x-x'|\bigg].
	\end{align*}
In addition, if $\alpha_\nu = 2$ then $\alpha = 2$.
\end{Lem}

\begin{proof}
Note that for $x = x'$, the claimed inequality is trivially true. From now on we suppose that $x\neq x'$. To obtain the required estimate, we consider
	\begin{align*}
	&\left| \phi_{\delta,\varepsilon}(y+xz)-\phi_{\delta,\varepsilon}(y+x'z)-(x-x')z\phi_{\delta,\varepsilon}'(y) \right|\\
	&\leq \left| \phi_{\delta,\varepsilon}(y+xz)-\phi_{\delta,\varepsilon}(y+x'z)-(x-x')z\phi_{\delta,\varepsilon}'(y+x'z) \right|
	+|x-x'||z|\left| \phi_{\delta,\varepsilon}'(y)-\phi_{\delta,\varepsilon}'(y+x'z) \right| =: A_z + B_z.
	\end{align*}	
Let $u\in (0,\infty)$. To estimate $A_z$ for $z \in (0,u)$, we apply a second order Taylor's expansion for $\phi_{\delta,\varepsilon}$ and use \eqref{phi4}. This gives
	\begin{align*}
	A_z \leq |x-x'|^2 |z|^2 \int_{0}^{1} \theta \phi_{\delta,\varepsilon}''(y+ \theta xz+(1-\theta)x'z)d \theta \leq |x-x'|^2 |z|^2 \frac{2\delta}{\varepsilon \log \delta}
\end{align*}	
while for $z \in (u,\infty)$, by the mean value theorem and \eqref{phi2}, we get 
	\begin{align*}
	A_z & \leq |x-x'| |z| \int_{0}^{1} \left|\phi'_{\delta,\varepsilon}(y+\theta xz+(1-\theta)x'z)-\phi'_{\delta,\varepsilon}(y) \right| d\theta
	\leq 2 |x-x'| |z|.
	\end{align*}
	

To this end, for some positive $k$ (which is chosen later) one considers the two cases $|x-x'| \leq k$ and $|x-x'|\geq k$. In the first case, take $u=1$ so that
$$
\int_0^\infty A_z \nu(dz) \leq C \left\{ |x-x'|^2 \frac{2\delta}{\varepsilon \log(\delta)} + |x-x'|\right\} \leq C_k \left\{ |x-x'|^{\alpha}\frac{2\delta}{\varepsilon \log(\delta)} + |x-x'|\right\} 
$$
\noindent for any $\alpha \in [0, 2]$.
In the second case $|x-x'|\geq k$, we select $u = |x-x'|^{-1} \in (0,k^{-1})$ and remark that
$$
\int_0^{|x-x'|^{-1}} z^2 \nu(dz) = |x-x'|^{\alpha-2} \frac{1}{|x-x'|^{\alpha-2}} \int_0^{|x-x'|^{-1}} z^2 \nu(dz) \leq |x-x'|^{\alpha-2} \sup_{\varepsilon \in [0,k^{-1}]}I^{1}_\varepsilon
$$
\noindent with $I^{1}_\varepsilon := \varepsilon^{\alpha-2} \int_0^{\varepsilon} z^2 \nu(dz)$. Note that $\lim_{\varepsilon \downarrow 0}I^{1}_\varepsilon = 0$ for any $\alpha > \alpha_\nu$ by \eqref{left:right:tail:asymptotics}. Similarly,
$$
\int_{|x-x'|^{-1}}^{\infty} z \nu(dz) = |x-x'|^{\alpha-1} \frac{1}{|x-x'|^{\alpha-1}} \int_{|x-x'|^{-1}}^{\infty} z \nu(dz) \leq |x-x'|^{\alpha-1} \sup_{\varepsilon \in [0,k^{-1}]}I^{2}_\varepsilon
$$
\noindent with $I^{2}_\varepsilon := \varepsilon^{\alpha-1} \int_{\varepsilon}^{\infty} z \nu(dz)$. Again we note that $\lim_{\varepsilon \downarrow 0}I^{2}_\varepsilon = 0$ for any $\alpha > \alpha_\nu$ by definition of $\alpha_\nu$. We thus conclude that for any positive constant $C$, one can pick $k$ sufficiently large, such that 
$$
\forall \alpha \in (\alpha_\nu, 2], \quad \int^\infty_0 A_z\, \nu(dz) \leq C \left\{|x-x'|^{\alpha}  \frac{2\delta}{\epsilon \log \delta} + |x-x'| \right\}.
$$

 We now deal with the term $B_z$. Let us first assume that $y x' \geq 0$. We perform a second order Taylor's expansion and employ \eqref{phi4} to obtain
\begin{align*}
	B_z & \leq |x'| |x-x'| |z|^2 \int_{0}^{1} \phi_{\delta,\varepsilon}''(y+ \theta x'z) \theta d\theta	\\
& \leq  2 \frac{|x'||x-x'|z^2}{\log(\delta)} \int_0^1   \frac{\1_{[\varepsilon/\delta,\varepsilon]}(|y+\theta x' z|)}{|y+\theta x'z|} \, \theta d\theta\\
&  \leq 2 \frac{|x'||x-x'|z^2 \1_{[0,\varepsilon)}(|y|)}{\log \delta}  \bigg( \frac{1}{|y|}\wedge \frac{\delta}{\varepsilon} \bigg)
\end{align*}

\noindent where for the last inequality we used the fact that $|y| \leq |y+\theta x'z|$ since $yx'\geq0$ and $z\geq0$.  Also, it is readily seen that $B_z \leq 2|x-x'||z|$. We thus conclude that if $y x'\geq0$, for any $u \in (0,\infty)$
$$
\int_{0}^{\infty} B_z \nu(dz) \leq |x'||x-x'| \frac{\1_{[0,\varepsilon)}(|y|)}{\log \delta}  \bigg( \frac{1}{|y|}\wedge \frac{\delta}{\varepsilon} \bigg) \int_{0}^{u} z^2 \nu(dz) + 2 |x-x'| \int_{u}^{\infty} z \nu(dz).
$$
We now treat the case $y x' < 0$. We split the $\nu(dz)$-integral into the two disjoint sets $\frac{|y|}{2|x'|}\wedge 1 < z$ and $\frac{|y|}{2|x'|}\wedge 1 \geq z$. In the case of small jumps, i.e. on the set $\frac{|y|}{2|x'|} \wedge 1 \geq z$, from the mean-value theorem and \eqref{phi4}, we obtain
\begin{align*}
B_z & = |x'||x-x'|z^2 \int_{0}^{1} \phi_{\delta,\varepsilon}''(y+ \theta x'z)d \theta	\\
& \leq 2 |x'||x-x'| z^2 \int_0^1 \frac{\1_{[\varepsilon / \delta,\varepsilon)}(|y+ \theta x'z|)}{|y+ \theta x'z|\log(\delta)} \, d\theta\\
& \leq 2 |x'||x-x'| z^2 \frac{\1_{[0,2\varepsilon)}(|y|)}{|y|\log(\delta)}
\end{align*}
where, for the last inequality, we used the fact that $yx' < 0$ and $\frac{|y|}{2|x'|} \wedge 1 \geq z$ imply
\begin{gather*}
|y+ \theta z x'| = |y(1 - \theta z|x'||y|^{-1}|)|  \geq \frac{|y|}{2}.
\end{gather*}
Now, observe that since $y x' < 0$, one has $0\leq -\sign(y) x' = |x'| \leq \kappa |y|$, which combined with the previous computations yield
$$
\int_0^{\frac{|y|}{2|x'|} \wedge 1} B_z \nu(dz) \leq 2 \kappa \frac{|x-x'|}{\log(\delta)} \int_{0}^{1} z^2 \nu(dz). 
$$
For large jumps, i.e. one the set $\frac{|y|}{2|x'|} \wedge 1 < z$, from \eqref{phi2}, we simply note that $B_z \leq 2|x-x'||z|$ so that
$$
\int_{\frac{|y|}{2|x'|} \wedge 1}^{\infty} B_z \nu(dz) \leq 2 |x-x'| \int_{\frac{|y|}{2|x'|} \wedge 1}^{\infty} z \nu(dz) \leq C |x-x'|
$$
\noindent where we used the facts that $|x'| \leq \kappa |y|$ and $\int_{1}^{\infty} z \nu(dz)$ for the last inequality. The proof is now complete.
\end{proof}

%

\end{document}